\documentclass[a4paper,12pt]{article}

\usepackage{soul}
\usepackage{latexsym}
\usepackage{amssymb}
\usepackage{amsmath}
\usepackage{amscd}
\usepackage{amsthm}
\usepackage{enumerate}
\theoremstyle{plain}
\newtheorem{theorem}{Theorem}[section]
\newtheorem{lemma}[theorem]{Lemma}
\newtheorem{proposition}[theorem]{Proposition}
\newtheorem{corollary}[theorem]{Corollary}
\newtheorem{definition}[theorem]{Definition}

\newtheorem{remark}[theorem]{Remark}

\theoremstyle{definition}
\newtheorem{romanremark}[theorem]{Remark}
\newtheorem{romanremarks}[theorem]{Remarks}
\newtheorem*{nonumremark}{Remark}
\newtheorem*{nonumremarks}{Remarks}

\newcommand\id{{\mathop{\rm id}}}

\newcommand\op{\mathop{\rm op}}

\newcommand\dd{{\mathop{\rm d}}}
\newcommand\hh{{\mathop{\rm h}}}

\newcommand{\sh}{{\mathop{\sigma\rm h}}}

\newcommand\mm{\mathop{\rm m}}

\newcommand\cb{\mathop{\rm cb}}

\newcommand\nph{\varphi}
\let\phi\nph
\let\epsilon\varepsilon

\newcommand\dist{\mathop{\rm dist}}

\newcommand{\As}[1][\A]{{#1}_1,\dots,{#1}_n}
\newcommand{\Hs}[1][H]{\As[#1]}

\newcommand{\opis}{\pi_1\otimess\pi_n}
\newcommand{\ff}{\mathit{f\!f}}
\newcommand{\HS}{\mathit{HS}}
\newcommand{\defeq}{\stackrel{\mathop{\rm def}}=}

\newcommand{\bb}[1]{\mathbb{#1}}
\newcommand{\bN}{\mathbb{N}}

\newcommand{\bC}{\mathbb{C}}
\newcommand{\CC}{\mathit{CC}}

\newcommand{\CB}{\mathit{CB}}

\newcommand{\CBS}{\CB^\sigma}

\newcommand{\haag}{\otimes_\hh}
\newcommand{\shaag}{\otimes_{\sh}}
\newcommand{\shaags}{\shaag\cdots\shaag}
\newcommand{\eh}{{\mathop{\rm eh}}}
\newcommand{\wh}{{\mathop{\rm w\mkern-1mu ^*\mkern-2mu h}}}

\newcommand{\ehaag}{\otimes_\eh}
\newcommand{\haags}{\haag\cdots\haag}
\newcommand{\ehaags}{\ehaag\cdots\ehaag}
\newcommand{\timess}{\times\cdots\times}
\newcommand{\otimess}{\otimes\cdots\otimes}
\newcommand{\odots}{\odot\cdots\odot}
\newcommand{\A}{\mathcal{A}}
\newcommand{\B}{\mathcal{B}}
\newcommand{\C}{\mathcal{C}}
\newcommand{\E}{\mathcal{E}}
\newcommand{\F}{\mathcal{F}}
\newcommand{\G}{\mathcal{G}}

\newcommand{\K}{\mathcal{K}}
\newcommand{\R}{\mathcal{R}}
\newcommand{\X}{\mathcal{X}}
\newcommand{\Y}{\mathcal{Y}}
\newcommand{\qand}{\quad\text{and}\quad}
\newcommand{\qtext}[1]{\quad\text{#1}}

\newcommand{\Kh}{\K_\hh}
\newcommand{\Bh}{\B_\hh}

\begin{document}

\title{Compactness properties of operator multipliers}

\author{K. Juschenko, R. H. Levene,\\ I. G. Todorov and L. Turowska}
\date{September 12, 2008}

\maketitle

\begin{abstract}
We continue the study of multidimensional operator multipliers
initiated in~\cite{jtt}. We introduce the notion of the symbol of
an operator multiplier. We characterise completely compact
operator multipliers in terms of their symbol as well as in terms
of approximation by finite rank multipliers. We give sufficient
conditions for the sets of compact and completely compact
multipliers to coincide and characterise the cases where an
operator multiplier in the minimal tensor product of two
C*-algebras is automatically compact. We give a description of
multilinear modular completely compact completely bounded maps
defined on the direct product of finitely many copies of the
C*-algebra of compact operators in terms of tensor products,
generalising results of Saar~\cite{saar}. \footnotetext{{\it
Primary:} 46L06, {\it Secondary:} 46L07, 47L25} \footnotetext{{\it
Keywords:} operator multiplier, complete compactness, Schur multiplier, Haagerup tensor product}
\end{abstract}

\section{Introduction}

A bounded function $\nph :
\mathbb{N}\times\mathbb{N}\rightarrow\mathbb{C}$ is called a Schur
multiplier if $(\nph(i,j)a_{ij})$ is the matrix of a bounded
linear operator on $\ell_2$ whenever $(a_{ij})$ is such. The study
of Schur multipliers was initiated by Schur in the early 20th
century and since then has attracted considerable attention, much
of which was inspired by A.~Grothendieck's characterisation of
these objects in his {\it
R$\acute{e}$sum$\acute{e}$}~\cite{grothendieck}. Grothendieck
showed that a function $\nph$ is a Schur multiplier precisely when
it has the form $\nph(i,j) = \sum_{k=1}^{\infty}a_k(i)b_k(j)$,
where $a_k,b_k : \bb{N}\rightarrow\bb{C}$ satisfy the conditions
$\sup_i \sum_{k=1}^{\infty} |a_k(i)|^2 < \infty$ and $\sup_j
\sum_{k=1}^{\infty} |b_k(j)|^2 < \infty$. In modern terminology,
this characterisation can be expressed by saying that $\nph$ is a
Schur multiplier precisely when it belongs to the extended
Haagerup tensor product $\ell_{\infty}\ehaag\ell_{\infty}$ of two
copies of $\ell_{\infty}$.

Special classes of Schur multipliers, e.g.~Toeplitz and Hankel Schur
multipliers, have played an important role in analysis and have been
studied extensively (see~\cite{Pi}). Compact Schur multipliers, that
is, the functions $\nph$ for which the mapping $(a_{ij})\rightarrow
(\nph(i,j)a_{ij})$ on $\B(\ell_2)$ is compact, were characterised
by Hladnik~\cite{hladnik}, who identified them with the elements of
the Haagerup tensor product $c_0\haag c_0$.

A non-commutative version of Schur multipliers was introduced by
Kissin and Shulman~\cite{ks} as follows. Let $\A$ and $\B$ be
$C^*$-algebras and let $\pi$ and $\rho$ be representations of $\A$
and $\B$ on Hilbert spaces $H$ and $K$, respectively. Identifying $H\otimes K$
with the Hilbert space $\C_2(H^{\dd},K)$ of all Hilbert-Schmidt
operators from the dual space $H^{\dd}$ of $H$ into $K$, we obtain
a representation $\sigma_{\pi,\rho}$ of the minimal tensor product
$\mathcal{A}\otimes \mathcal{B}$ acting on $\C_2(H^{\dd},K)$. An
element $\varphi\in\mathcal{A}\otimes\mathcal{B}$ is called a
$\pi$,$\rho$-multiplier if $\sigma_{\pi,\rho}(\varphi)$ is bounded
in the operator norm of $\C_2(H^{\dd},K)$. If $\varphi$ is a
$\pi$,$\rho$-multiplier for any pair of representations
$(\pi,\rho)$ then $\varphi$ is called a universal (operator)
multiplier.

Multidimensional Schur multipliers and their non-commutative versions
were introduced and studied in~\cite{jtt}, where the authors gave, in
particular, a characterisation of universal multipliers as certain
weak limits of elements of the algebraic tensor product of the
corresponding $C^*$-algebras, generalising the corresponding results
of Grothendieck and Peller~\cite{grothendieck,peller} as previously
conjectured by Kissin and Shulman in~\cite{ks}. Let $\As$ be
C*-algebras.  Like Schur multipliers, elements of the set $M(\As)$ of
(multidimensional) universal multipliers give rise to completely
bounded (multilinear) maps.  Requiring these maps to be compact or
completely compact, we define the sets of compact and completely
compact operator multipliers denoted by $M_{c}(\As)$ and
$M_{cc}(\As)$, respectively. The notion of complete compactness we use
is an operator space version of compactness which was introduced by
Saar~\cite{saar} and subsequently studied by Oikhberg~\cite{oik} and
Webster~\cite{webster}. Our results on operator multipliers rely on
the main result of Section~\ref{completelyc} where we prove a
representation theorem for completely compact completely bounded
multilinear maps. In~\cite{christensen_sinclair} Christensen and
Sinclair established a representation result for completely bounded
multilinear maps which implies that every such map $\Phi :
\K(H_2,H_1)\haags\K(H_n,H_{n-1})\to\K(H_n,H_1)$ 
(where, for Hilbert spaces $H'$ and $H''$, we denote by
$\K(H',H'')$ the space of all compact operators from $H'$ into
$H''$) has the form
\begin{equation}\label{formo}
\Phi(x_1\otimess x_{n-1})=A_1(x_1\otimes
1)A_2\ldots(x_{n-1}\otimes 1)A_{n},
\end{equation}
for some index set $J$ and bounded block operator matrices $A_1\in M_{1,J}(\B(H_1))$, $A_2\in M_J(\B(H_2)), \ldots,
A_n\in M_{J,1}(\B(H_n))$. In other words, $\Phi$ arises from an element
$$u=A_1\odots A_n\in \B(H_1)\ehaags\B(H_n)$$
of the extended Haagerup tensor product of
$\B(H_1),\dots,\B(H_n)$. Moreover, if $\Phi$ is
$\A_1',\ldots\A_n'$-modular for some von Neumann algebras 
$\A_1,\dots,\A_n$, then the entries of $A_i$ can be chosen from
$\A_i$. We show in Section~\ref{completelyc} that a map $\Phi$
as above is completely compact precisely when it has a
representation of the form (\ref{formo}) where
$$u=A_1\odots A_n\in \K(H_1)\haag(\B(H_2)\ehaags\B(H_{n-1}))\haag\K(H_n).$$
This extends a result of Saar~\cite{saar} in the two dimensional
case. If, additionally, $\As$ are von Neumann algebras and $\Phi$
is $\A_1',\ldots\A_n'$-modular then $u$ can be chosen from
$\K(\A_1)\otimes_{\hh}(\A_2\ehaags\A_{n-1})\otimes_{\hh}\K(\A_n)$,
where $\K(\A)$ denotes the ideal of compact elements of a
$C^*$-algebra $\A$. As a consequence of this and a result of
Effros and Kishimoto~\cite{effros_kishimoto} we point out the
completely isometric identifications
\[\CC(\K(H_2,H_1))^{**}\simeq(\K(H_1)\haag\K(H_2))^{**}\simeq
\CB(\B(H_2,H_1)),\] where $\CC(\X)$ and $\CB(\X)$ are the spaces
of completely compact and completely bounded maps on an operator
space ${\X}$, respectively.

In Section~\ref{sec:cbmult} we pinpoint the connection between
universal operator multipliers and completely bounded maps. This
technical result is used in Section~\ref{ssy} to define the symbol
$u_{\varphi}$ of an operator multiplier $\varphi\in M(\As)$ which, in
the case $n$ is even (resp.~odd) is an element of
$\A_n\ehaag\A_{n-1}^o\ehaags\A_1^o$
(resp.~$\A_n\ehaag\A_{n-1}^o\ehaags\A_1$). Here $\A^o$ is the
opposite C*-algebra of $\A$. This notion extends a similar notion that
was given in the case of completely bounded masa-bimodule maps by
Katavolos and Paulsen in~\cite{kp}. We give a symbolic calculus for
universal multipliers which is used to establish a universal property
of the symbol related to the representation theory of the C*-algebras
under consideration.

The symbol of a universal multiplier is used in Section~\ref{schar} to
single out the completely compact multipliers within the set of
all operator multipliers. In fact, we show that $\varphi\in
M_{cc}(\As)$ if and only if
\[u_{\phi}\in
\begin{cases}
\K(\A_n) \haag(\A_{n-1}^o\ehaags\A_2)\haag \K(\A_1^o)\;&\text{if $n$ is even}\\
\K(\A_n) \haag(\A_{n-1}^o\ehaags\A_2^o)\haag \K(\A_1)&\text{if $n$
is odd,}
\end{cases}\] which is equivalent to the approximability of $\phi$ in
the multiplier norm by operator multipliers of finite rank whose range
consists of finite rank operators. It follows that a multidimensional
Schur multiplier $\nph\in \ell_{\infty}(X_1\times\dots\times X_n)$ is
compact if and only if 
$\nph\in c_0(X_1)\haag(\ell_{\infty}(X_2)\ehaags\ell_{\infty}(X_{n-1}))
\haag c_0(X_n)$.

In Section~\ref{sec:cccm} we use Saar's construction~\cite{saar}
of a completely bounded compact mapping which is not completely
compact to show that %
the
inclusion $M_{cc}(\As)\subseteq M_{c}(\As)$ is proper
if both $\K(\A_1)$ and $\K(\A_n)$ contain
full matrix algebras of arbitrarily large sizes. However, if
both $\K(\A_1)$ and $\K(\A_n)$ are isomorphic to a $c_0$-sum of
matrix algebras of uniformly bounded sizes then the sets of
compact and completely compact multipliers coincide. The case when
only one of $\K(\A_1)$ and $\K(\A_n)$ contains matrix algebras of
arbitrary large size remains, however, unsettled. Finally, for $n
= 2$, we characterise the cases where every universal multiplier
is automatically compact: this happens precisely when one of the
algebras $\A_1$ and $\A_2$ is finite dimensional and the other one
coincides with its algebra of compact elements.

\bigskip

\noindent {\bf Acknowledgements} We are grateful to V.S. Shulman
for stimulating results, questions and discussions. We would like
to thank M. Neufang for pointing out to us
Corollary~\ref{identification} and R. Smith for a discussion
 concerning Remark \ref{rsmith}. The first named
author is grateful to Gilles Pisier for the support of the one
semester visit to the University of Paris 6 and the warm
atmosphere at the department, where one of the last drafts of the
paper was finished.

The first named author was supported by The Royal Swedish Academy
of Sciences, Knut och Alice Wallenbergs stiftelse and
Jubileumsfonden of the University of Gothenburg's Research
Foundation. The second and the third named authors were supported
by Engineering and Physical Sciences Research Council grant
EP/D050677/1. The last named author was supported by the Swedish
Research Council.

\section{Preliminaries}\label{sec:prelims}

We start by recalling standard notation and notions from operator
space theory. We refer the reader
to~\cite{blm},~\cite{effros_ruan1},~\cite{paulsen-book}
and~\cite{pisier-book} for more details.

If $H$ and $K$ are Hilbert spaces we let $\B(H,K)$
(resp.~$\K(H,K)$) denote the set of all bounded linear
(resp.~compact) operators from $H$ into $K$. If $I$ is a set we
let $H^I$ be the direct sum of $|I|$ copies of $H$ and set
$H^{\infty} = H^{\bb{N}}$. An operator space $\E$ is a closed
subspace of $\B(H,K)$, for some Hilbert spaces $H$ and $K$. The
opposite operator space $\E^o$ associated with $\E$ is the space
$\E^o = \{x^{\dd} : x\in\E\}\subseteq\B(K^{\dd},H^{\dd})$. Here,
and in the sequel, $H^{\dd}=\{\xi^\dd:\xi\in H\}$ denotes the dual
of the Hilbert space $H$, where $\xi^\dd(\eta)=(\eta,\xi)$ for
$\eta\in H$. Note that $H^{\dd}$ is canonically conjugate-linearly
isometric to $H$. We also adopt the notation $x^\dd\in
\B(K^\dd,H^\dd)$ for the Banach space adjoint of $x\in\B(H,K)$, so
that $x^\dd \xi^\dd=(x^*\xi)^\dd$ for $\xi\in K$.  As usual,
$\E^*$ will denote the operator space dual of $\E$. If
$n,m\in\bb{N}$, by $M_{n,m}(\E)$ we denote the space of all $n$ by
$m$ matrices with entries in $\E$ and let $M_n(\E) = M_{n,n}(\E)$.
The space $M_{n,m}(\E)$ carries a natural norm arising from the
embedding $M_{n,m}(\E)\subseteq \B(H^m,K^n)$. Let $I$ and $J$ be
arbitrary index sets. If $v$ is a matrix with entries in $\E$
and indexed by $I\times J$, and $I_0\subseteq I$ and $J_0\subseteq
J$ are finite sets, we let $v_{I_0,J_0}\in M_{I_0,J_0}(\E)$ be the
matrix obtained by restricting $v$ to the indices from $I_0\times
J_0$. We define $M_{I,J}(\E)$ to be the space of all such $v$ for
which
\[\|v\|\defeq \sup\{\|v_{I_0,J_0}\| : \text{$I_0\subseteq I$,
$J_0\subseteq J$ finite}\} < \infty.\] Then $M_{I,J}(\E)$ is an
operator space~\cite[\S 10.1]{effros_ruan1}. Note that
$M_{I,J}(\B(H,K))$ can be naturally identified with $\B(H^J,K^I)$
and every $v\in M_{I,J}(\B(H,K))$ is the weak limit of
$\{v_{I_0,J_0}\}$ along the net $\{(I_0,J_0) : I_0\subseteq I,
J_0\subseteq J \mbox{ finite}\}$. We set $M_{I}(\E) =
M_{I,I}(\E)$. For $A=(a_{ij})\in M_I(\E)$, we write
$A^\dd=(a_{ij}^\dd)\in M_I(\E^o)$.

\subsection{Completely bounded maps and Haagerup tensor products}

If $\E$ and $\F$ are operator spaces, a linear map $\Phi:\E\to \F$
is called completely bounded if the maps $\Phi^{(k)}:M_k(\E)\to
M_k(\F)$ given by $\Phi^{(k)}((a_{ij}))=(\Phi(a_{ij}))$ are
bounded for every $k\in {\mathbb N}$ and $\|\Phi\|_{\cb}
\defeq \sup_k\|\Phi^{(k)}\| < \infty$.

Given linear spaces $\E_1,\dots,\E_n$, we denote by
$\E_1\odots\E_n$ their algebraic tensor product. If 
$\E_1,\dots,\E_n$ are operator spaces and $a^k=(a_{ij}^k)\in
M_{m_k,m_{k+1}}(\E_k)$, $m_k\in\mathbb{N}$, $k=1,\dots,n$, we
define the multiplicative product
$$a^1\odots a^n\in M_{m_1,m_{n+1}}(\E_1\odots\E_n)$$ by letting
its $(i,j)$-entry $(a^1\odots a^n)_{ij}$ be
$\sum_{i_2,\dots,i_n}a_{i,i_2}^1\otimes a_{i_2,i_3}^2\otimess
a^n_{i_n,j}.$ If $\E$ is another operator space and
$\Phi:\E_1\timess\E_n\to\E$ is a multilinear map we let
$$\Phi^{(m)}: M_m(\E_1)\times\dots\times M_m(\E_n)\to M_m(\E)$$
be the map given by
$$\big(\Phi^{(m)}(a^1,\dots,a^n)\big)_{ij} =
\sum_{i_2,\dots,i_n}\Phi(a_{i,i_2}^1,a_{i_2,i_3}^2,\dots,a_{i_n,j}^n),$$
where $a^k = (a^k_{s,t})\in M_{m}(\E_k)$, $k = 1,\dots,n$.
The multilinear map $\Phi$ is called completely bounded if there
exists a constant $C > 0$ such that, for all $m\in{\mathbb N}$,
$$\|\Phi^{(m)}(a^1,\dots,a^n)\|\leq C\|a^1\|\dots\|a^n\|, \ \   a^k\in M_{m}(\E_k), \ k = 1,\dots,n.$$
Set $\|\Phi\|_{\cb}\defeq\sup\{\|\Phi^{(m)}(a^1,\dots,a^n)\| :
m\in\bb{N}, \|a^1\|,\dots,\|a^n\|\leq 1\}$. It is well-known
(see~\cite{effros_ruan1,paulsen-smith}) that a completely bounded
multilinear map $\Phi$ gives rise to a completely bounded map on
the Haagerup tensor product ${\E_1}\haags\E_n$
(see~\cite{effros_ruan1} and~\cite{pisier-book} for its definition
the and basic properties).

The set of all completely bounded multilinear maps from
$\E_1\timess\E_n$ into $\E$ will be denoted by
$\CB(\E_1\timess\E_n,\E)$. If $\E_1,\dots,\E_n$ and $\E$ are dual
operator spaces we say that a map $\Phi\in
\CB(\E_1\timess\E_n,\E)$ is normal~\cite{christensen_sinclair} if
it is weak* continuous in each variable. We write
$\CBS(\E_1\timess\E_n,\E)$ for the space of all normal maps in
$\CB(\E_1\timess\E_n,\E)$.

The {\it extended Haagerup tensor product} $\E_1\ehaags\E_n$ is
defined~\cite{effros_ruan} as the space of all normal completely
bounded maps $u:\E_1^*\timess\E_n^*\to{\mathbb C}$. It was shown
in~\cite{effros_ruan} that if $u\in \E_1\ehaags\E_n$ then there
exist index sets $J_1,J_2,\dots,J_{n-1}$ and matrices $a^1 =
(a^1_{1,s})\in M_{1,J_1}(\E_1)$, $a^2 = (a^2_{s,t})\in
M_{J_1,J_2}(\E_2),\dots, a^n = (a^n_{t,1})\in M_{J_{n-1},1}(\E_n)$
such that if $f_i\in\E_i^*$, $i=1,\dots,n$, then
\begin{equation}\label{product}
\langle u,f_1\otimess
f_n\rangle\defeq u(f_1,\dots,f_n) = \langle
a^1,f_1\rangle\dots\langle a^n,f_n\rangle,
\end{equation}
where $\langle a^k,f_k\rangle=\big(f_k(a_{s,t}^k)\big)_{s,t}$ and
the product of the (possibly infinite) matrices in~\eqref{product}
is defined to be the limit of the sums
$$\sum_{i_1\in F_1,\dots, i_{n-1}\in F_{n-1}}f_1(
a_{1,i_1}^1)f_2(a^2_{i_1,i_2})\dots f_n(a_{i_{n-1},1}^n)$$ along
the net $\{(F_1\timess F_{n-1}) : F_j\subseteq J_j \mbox{ finite},
1\leq j\leq n-1\}$. We may thus identify $u$ with the matrix
product $a^1\odots a^n$; two elements $a^1\odots a^n$ and $\tilde
a^1\odots \tilde a^n$ coincide if $\langle
a^1,f_1\rangle\dots\langle a^n,f_n\rangle = \langle \tilde
a^1,f_1\rangle\dots\langle \tilde a^n,f_n\rangle$ for all
$f_i\in\E_i^*$. Moreover,
$$\|u\|_\eh = \inf\{\|a^1\|\dots\|a^n\| : u = a^1\odots
a^n\}.$$ The space $\E_1\ehaags\E_n$ has a natural operator space
structure~\cite{effros_ruan}. If $\E_1,\dots,\E_n$ are dual
operator spaces then by~\cite[Theorem~5.3]{effros_ruan}
$\E_1\ehaags\E_n$ coincides with the weak* Haagerup tensor product
$\E_1\otimes_{\wh}\dots\otimes_{\wh}\E_n$ of Blecher and
Smith~\cite{blecher_smith}. Given operator spaces $\F_i$ and
completely bounded maps $g_i:\E_i\to \F_i$, $i = 1,\dots,n$,
Effros and Ruan~\cite{effros_ruan} define a completely bounded map
\begin{align*}
  g=g_1\ehaags
  g_n:\E_1\ehaags\E_n&\to \F_1\ehaags \F_n,\\
  a^1\odots a^n&\mapsto
\langle a^1,g_1\rangle \odots \langle a^n,g_n\rangle
\end{align*}
where $\langle a^k,g_k\rangle=\big(g_k(a^k_{ij})\big)$. Thus
\begin{equation}
  \label{eq:eh-adjoints}
  \langle g(u),f_1\otimess f_n\rangle = \langle u, (f_1\circ g_1)\otimess
  (f_n\circ g_n)\rangle
\end{equation}
for $u\in \E_1\ehaags\E_n$ and $f_i\in \F_i^*$, $i =
1,\dots,n$.

The following fact is a straightforward consequence of a
well-known theorem due to Christensen and
Sinclair~\cite{christensen_sinclair}, and it will be used
throughout the exposition.

\begin{theorem}\label{cs}
Let $H_i$ be a Hilbert space and $\R_i\subseteq\B(H_i)$ be a
von Neumann algebra, $i = 1,\dots,n$. There exists an isometry
$\gamma$ from $\R_1\ehaags \R_n$ onto the space
of all $\R_1',\dots,\R_n'$-modular maps in
$CB^{\sigma}(\B(H_2,H_1)\timess \B(H_n,H_{n-1}),\B(H_n,H_1))$,
given as follows: if $u = A_1\odot \dots\odot A_n \in 
\R_1\ehaags \R_n$ then
$$\gamma(u)(T_1,\dots,T_{n-1}) = A_1 (T_1\otimes I)A_2\dots
A_{n-1}(T_{n-1}\otimes I)A_n,$$ for all $T_i\in \B(H_{i+1},H_i)$, $i = 1,\dots,n-1$.
\end{theorem}

We now turn to the definition of slice maps which will play an
important role in our proofs. Given $\omega_1\in \B(H_1)^*$ we set $L_{\omega_1} =
\omega_1\otimes\id_{\B(H_2)}$.
After identifying $\bb{C}\otimes\B(H_2)$ with $\B(H_2)$ we obtain a mapping called a left slice
map $L_{\omega_1} : \B(H_1)\ehaag\B(H_2)\rightarrow \B(H_2)$. Similarly, for $\omega_2\in
\B(H_2)^*$ we obtain a right slice map $R_{\omega_2}:\B(H_1)\ehaag\B(H_2)\to\B(H_1)$.
If $u=\sum_{i\in I}v_i\otimes w_i\in\B(H_1)\ehaag\B(H_2)$ where $v =
(v_i)_{i\in I}\in M_{1,I}(\B(H_1))$ and $w = (w_i)_{i\in I}\in
M_{I,1}(\B(H_2))$, then
\[L_{\omega_1}(u)=\sum_{i\in I}\omega_1(v_i)w_i \qand
R_{\omega_2}(u)=\sum_{i\in I}\omega_2(w_i)v_i.\]
Moreover,
\begin{equation}
  \label{eq:slices}
  \langle R_{\omega_2}(u), \omega_1\rangle =
  \langle u, \omega_1\otimes
  \omega_2\rangle=\langle L_{\omega_1}(u), \omega_2\rangle = \sum_{i\in I}
\omega_1(v_i)\omega_2(w_i).
\end{equation}

It was shown in~\cite{spronk} that if $\E\subseteq \B(H_1)$
and $\F\subseteq\B(H_2)$ are closed subspaces then, up to a
complete isometry,
\begin{align}\label{eq:spronk}
  \begin{split}
    \E\ehaag\F &= \{u\in \B(H_1)\ehaag\B(H_2) :
    L_{\omega_1}(u)\in\F\text{ and } R_{\omega_2}(u)\in\E
    \\&\hspace{4cm}\mbox{ for all } \omega_1\in \B(H_1)_*\text{ and } \omega_2\in\B(H_2)_*\}
  \end{split}\\[6pt]
  \begin{split}
    &=\{u\in \B(H_1)\ehaag\B(H_2) :
    L_{\omega_1}(u)\in\F\text{ and } R_{\omega_2}(u)\in\E
    \\&\hspace{4cm}\mbox{ for all } \omega_1\in \B(H_1)^*\text{ and } \omega_2\in\B(H_2)^*\}.\notag
  \end{split}
\end{align}
Moreover~\cite{smith},
\begin{align}\label{eq:spronk2}
\begin{split}
  \E\haag\F&=\{u\in \B(H_1)\haag\B(H_2): L_{\omega_1}(u)\in\F\text{ and } R_{\omega_2}(u)\in\E \\&\hspace{4cm}\mbox{ for
    all } \omega_1\in \B(H_1)^*\text{ and } \omega_2\in\B(H_2)^*\}.
\end{split}
\end{align}

Thus, $\E\haag\F$ can be canonically identified with a subspace of
$\B(H_1)\haag\B(H_2)$ which, on the other hand, sits completely
isometrically in $\B(H_1)\ehaag\B(H_2)$. These identifications are
made in the statement of the following lemma which will be useful
for us later.
\begin{lemma}\label{lem:ss}
  If $H_1,H_2,H_3$ are Hilbert spaces and $\E_1,\E_2\subseteq\B(H_1)$,
  $\F_1,\F_2\subseteq\B(H_2)$ and $\G_1,\G_2\subseteq\B(H_3)$ are
  operator spaces, then
  \begin{gather*}
    (\E_1\ehaag\F_1)\cap(\E_2\haag\F_2)=(\E_1\cap\E_2)\haag(\F_1\cap\F_2)\quad
    \text{and}\\
    (\E_1\ehaag\F_1\ehaag\G_1)\cap(\E_2\haag\F_2\haag\G_2)=(\E_1\cap\E_2)\haag(\F_1\cap\F_2)\haag(\G_1\cap\G_2)
  \end{gather*}
  completely isometrically.
\end{lemma}
\begin{proof}
  Since $\ehaag$ and $\haag$ are both associative, the second equation
  follows from the first.
  If $u\in (\E_1\ehaag\F_1)\cap(\E_2\haag\F_2) \subseteq
  \B(H_1)\haag\B(H_2)$ then $L_\phi(u)\in \F_1\cap\F_2$ and
  $R_\psi(u)\in \E_1\cap \E_2$ whenever $\phi\in \B(H_1)^*$ and $\psi\in
  \B(H_2)^*$. By (\ref{eq:spronk2}), $u\in
  (\E_1\cap\E_2)\haag(\F_1\cap\F_2)$. The converse inclusion follows
  immediately in light of the injectivity of the
  Haagerup tensor product.
\end{proof}

\subsection{Operator multipliers}
We now recall some definitions and results from~\cite{ks}
and~\cite{jtt} that will be needed later. Let $\Hs$ be Hilbert
spaces and $H = H_1\otimess H_n$ be their Hilbertian tensor product.
Set $\HS(H_1,H_2) =
\C_2(H_1^{\dd},H_2)$ and let $\theta_{H_1,H_2} : H_1\otimes
H_2\rightarrow \HS(H_1,H_2)$ be the canonical isometry given by
$\theta (\xi_1\otimes \xi_2)(\eta^{\dd}) = (\xi_1,\eta)\xi_2$ for
$\xi_1,\eta\in H_1$ and $\xi_2\in H_2$. When $n$ is
even, we inductively define
$$\HS(\Hs) \defeq \C_2(\HS(H_2,H_3)^{\dd},
\HS(H_1,H_4,\dots,H_n)),$$ and let
$\theta_{\Hs} : H \rightarrow \HS(\Hs)$
be given by $$\theta_{\Hs}(\xi_{2,3}\otimes
\xi)=\theta_{\HS(H_2,H_3),\HS(H_1,H_4,\dots,H_n)}(\theta_{H_2,H_3}(\xi_{2,3})\otimes
\theta_{H_1,H_4,\dots,H_n}(\xi)),$$ where $\xi_{2,3}\in H_2\otimes
H_3$ and $\xi\in H_1\otimes H_4\otimess H_n$. When $n$ is odd, we let
$$\HS(\Hs) \defeq  \HS(\bb{C},\Hs).$$
If $K$ is a Hilbert space, we will identify $\C_2(\bb{C}^{\dd},K)$ with $K$ via the map $S\rightarrow S(1^{\dd})$.
The isomorphism $\theta_{\Hs}$ in the odd case is given by
$$\theta_{\Hs}(\xi) =
\theta_{\bb{C},\Hs}(1\otimes\xi).$$ We will omit the subscripts
when they are clear from the context and simply write $\theta$.

If $\xi\in H_1\otimes H_2$ we let $\|\xi\|_{\op}$ denote the
operator norm of $\theta(\xi)$. By $\|\cdot\|_2$ we will denote the
Hilbert-Schmidt norm.

Let
$$\Gamma(\Hs) =
\begin{cases}
(H_1\otimes H_2)\odot (H_2\otimes H_3)^{\dd}\odot\dots\odot (H_{n-1}\otimes H_n)\;&\text{$n$ even,}\\
(H_1\otimes H_2)^{\dd}\odot (H_2\otimes H_3)\odot\dots\odot
(H_{n-1}\otimes H_n)&\text{$n$ odd.}
\end{cases}$$
We equip $\Gamma(\Hs)$ with the Haagerup norm $\|\cdot\|_{\hh}$
where each of the terms of the algebraic tensor product is given
the opposite operator space structure to the one arising from the
embedding $H\otimes K\hookrightarrow
(\C_2(H^{\dd},K),\|\cdot\|_{\op})$. We denote by
$\|\cdot\|_{2,\wedge}$ the projective norm on $\Gamma(\Hs)$ where
each of the terms is given its Hilbert space norm.

Suppose $n$ is even. For each $\nph\in \B(H)$ we let $S_{\nph} :
\Gamma(\Hs)\rightarrow \B(H_1^{\dd},H_n)$ be the map
given by 
$$S_{\nph}(\xi) =
\theta(\nph(\xi_{1,2}\otimes\xi_{3,4}\otimess\xi_{n-1,n}))(\theta(\eta_{2,3}^{\dd}))
(\theta(\eta_{4,5}^{\dd}))\dots(\theta(\eta_{n-2,n-1}^{\dd}))$$
where  $\zeta =
\xi_{1,2}\odot\eta_{2,3}^{\dd}\dots\odot\xi_{n-1,n}\in
\Gamma(\Hs)$ is an elementary tensor.
In particular, if $A_i\in \B(H_i)$, $i=1,\ldots,n$, and $\varphi=A_1\otimes\ldots\otimes A_n$
 then
$$S_{\nph}(\zeta)=A_n\theta(\xi_{n-1,n})\ldots A_3^{\dd}\theta(\eta_{2,3}^{\dd})A_2\theta(\xi_{1,2})A_1^{\dd}.$$

Now
suppose that $n$ is odd and let $\zeta\in\Gamma(\Hs)$ and
$\xi_1\in H_1$. Then
$$\xi_1\otimes\zeta\in H_1\odot\Gamma(\Hs) =
\Gamma(\bb{C},\Hs).$$ For $\phi\in \B(H)$ we let $S_{\nph}(\zeta)$ be the
operator defined on $H_1$ by
$$S_{\nph}(\zeta)(\xi_1) = S_{1\otimes\nph}(\xi_1\otimes\zeta).$$
Note that $S_{1\otimes\nph}(\xi_1\otimes\zeta)\in
\C_2(\bb{C}^d,H_n)$; thus, $S_{\nph}(\zeta)(\xi_1)$ can be viewed as
an element of $H_n$. It was shown in~\cite{jtt} that
$S_{\nph}(\zeta) \in \B(H_1,H_n)$. If
$\zeta=\eta_{1,2}^{\dd}\otimes\xi_{2,3}\otimes\ldots\otimes\xi_{n-1,n}$
and $\varphi=A_1\otimes\ldots\otimes A_n$ for $A_i\in \B(H_i)$,
$i=1,\ldots,n$ then
$$S_{\nph}(\zeta)=A_n\theta(\xi_{n-1,n})\ldots
A_3\theta(\xi_{2,3})A_2^{\dd}\theta(\eta_{1,2}^{\dd})A_1.$$

As observed in \cite[Remark~4.3]{jtt}, for any $\phi\in
\B(H)$ and $\zeta\in\Gamma(\Hs)$,
\begin{equation}
  \label{rk4.3}
  \|S_\phi(\zeta)\|_{\op}\leq \|\phi\|\,\|\zeta\|_{2,\wedge}.
\end{equation}
On the other hand, an element $\nph\in \B(H)$ is called a {\it
  concrete operator multiplier} if there exists $C > 0$ such that
$\|S_{\nph}(\zeta)\|_{\op}\leq C\|\zeta\|_{\hh}$ for each $\zeta\in
\Gamma(\Hs)$. When $n=2$, this is equivalent to
$\|S_{\nph}(\zeta)\|_{\op}\leq C\|\theta(\zeta)\|_{\op}$ for each
$\zeta\in H_1\otimes H_2$. We call the smallest constant $C$ with this
property the concrete multiplier norm of $\nph$.

Now let $\A_i$ be a C*-algebra and $\pi_i :
\A_i\rightarrow\B(H_i)$ be a representation, $i = 1,\dots,n$. Set
$\pi = \pi_1\otimess\pi_n : \A_1\otimess \A_n\rightarrow
\B(H_1\otimess H_n)$ (here, and in the sequel, by $\A\otimes \B$
we will denote the minimal tensor product of the C*-algebras $\A$
and $\B$). An element $\nph\in \A_1\otimess\A_n$ is called a {\it
$\pi_1,\dots,\pi_n$-multiplier} if $\pi(\nph)$ is a concrete
operator multiplier. We denote by $\|\nph\|_{\pi_1,\dots,\pi_n}$
the concrete multiplier norm of $\pi(\nph)$. We call $\nph$ a {\it
universal multiplier} if it is a $\pi_1,\dots,\pi_n$-multiplier
for all representations $\pi_i$ of $\A_i$, $i = 1,\dots,n$. We
denote the collection of all universal multipliers by $M(\As)$;
from this definition, it immediately follows that
\[ \A_1\odots\A_n\subseteq M(\As)\subseteq \A_1\otimess\A_n.\]
It was observed in~\cite{jtt} that if $\phi\in M(\As)$ then
\[\|\phi\|_{\mm}\defeq\sup\{\|\nph\|_{\pi_1,\dots,\pi_n}:
  \text{$\pi_i$ is a representation of $\A_i$, $i=1,\dots,n$}\}<\infty.\]
It is obvious that if $\A_i$ and
$\B_i$ are C*-algebras and $\rho_i : \A_i\rightarrow\B_i$ is a
$*$-isomorphism, $i = 1,\dots,n$, then
$$(\rho_1\otimess\rho_n)(M(\As)) = M(\B_1,\dots,\B_n).$$

If $\nph$ is an operator, and $\{\nph_{\nu}\}$ a net of operators,
acting on $H_1\otimes\dots\otimes H_n$ we say that $\{\nph_{\nu}\}$
converges semi-weakly to $\nph$ if
$(\nph_{\nu}\xi,\eta)\rightarrow_{\nu} (\nph\xi,\eta)$ for all
$\xi,\eta\in H_1\odot\dots\odot H_n$. The following characterisation
of universal multipliers was established in~\cite{jtt} (see
Theorem~6.5, the subsequent remark and the proof of Proposition 6.2)
and will be used extensively in the sequel.

\begin{theorem}\label{jtt}
Let $\A_i\subseteq\B(H_i)$ be a C*-algebra, $i = 1,\dots,n$,
and $\nph\in \A_1\otimess\A_n$. Suppose that $n$
is even. The following are equivalent:

(i) \ $\nph\in M(\As)$;

(ii) there exists a net $\{\nph_{\nu}\}$ where $\nph_{\nu} =
A_1^{\nu}\odot A_2^{\nu}\odot\dots\odot A_n^{\nu}$ and $A_i^{\nu}$
is a finite block operator matrix with entries in $\A_i$ such
that $\nph_{\nu}\rightarrow\nph$ semi-weakly, 
$\|\nph_{\nu}\|_{\mm}\leq\Pi\|A_{2i}^{\nu}\|\Pi\|A_{2i+1}^{\nu\dd}\|$ and the operator
norms $\|A_i^{\nu}\|$ for $i$ even and $\|A_i^{\nu}{}^\dd\|$ for
$i$ odd, are bounded by a constant depending only on $n$.

For every net $\{\nph_{\nu}\}$ satisfying (ii) we have that
$S_{\nph_{\nu}}(\zeta)\rightarrow S_{\nph}(\zeta)$ weakly for all
$\zeta = \xi_{1,2}\otimes\dots\otimes\xi_{n-1,n}
\in\Gamma(H_1,\dots,H_n)$ and $\sup_\nu\|\nph_\nu\|_{\mm}$ is finite.

Moreover, the net $\nph_{\nu}$ can be chosen in (ii) so that
$A_i^{\nu}\rightarrow A_i$ (resp.~$A_i^{\nu}{}^{\dd}\rightarrow
A_i^{\dd}$) strongly for $i$ even (resp.~for $i$ odd) for some
bounded block operator matrix $A_i$ with entries in $\A_i''$
(resp.~$(\A_i^{\dd})''$) such that
$$S_{\id\otimes\dots\otimes\id(\nph)}(\zeta) =
A_n(\theta(\xi_{n-1,n})\otimes I)\dots(\theta(\xi_{1,2})\otimes
I)A_1^{\dd},$$ for all $\zeta =
\xi_{1,2}\otimes\dots\otimes\xi_{n-1,n} \in\Gamma(H_1,\dots,H_n)$.

A similar statement holds if $n$ is odd.
\end{theorem}

Finally, recall that an element $a$ of a C*-algebra $\A$ is
called {\it compact} if the operator $x\rightarrow axa$ on $\A$ is
compact. Let $\K(\A)$ be the collection of all compact elements of
$\A$. It is well known~\cite{erdos,ylinen} that $a\in
\K(\A)$ if and only if there exists a faithful representation
$\pi$ of $\A$ such that $\pi(a)$ is a compact operator. Moreover,
$\pi$ can be taken to be the reduced atomic representation of
$\A$. The notion of a compact element of a C*-algebra will play a
central role in Sections~\ref{schar} and~\ref{sec:cccm} of the
paper.

\section{Completely compact maps}\label{completelyc}

We start by recalling the notion of a completely compact map
introduced in~\cite{saar} and studied further in~\cite{webster} and~\cite{oik}.
By way of motivation, recall that if $\X$ and $\Y$ are Banach
spaces then a bounded linear map $\Phi:\X\to\Y$ is compact if and
only if for every $\epsilon>0$, there exists a finite dimensional
subspace $F\subseteq\Y$ such that $\dist(\Phi(x),F)<\epsilon$ for
every $x$ in the unit ball of~$\X$.

Now let $\X$ and $\Y$ be operator spaces. A completely
bounded map $\Phi : \X\rightarrow\Y$ is called {\it completely
compact} if for each $\epsilon > 0$ there exists a finite
dimensional subspace $F\subseteq\Y$ such that
\[\dist(\Phi^{(m)}(x), M_m(F)) < \epsilon,\] for every
$x\in M_{m}(\X)$ with $\|x\|\leq 1$ and every $m\in \bb{N}$. We
extend this definition to multilinear maps: if $\Y,
\X_1,\dots,\X_n$ are operator spaces and $\Phi :
\X_1\timess\X_n\rightarrow\Y$ is a completely bounded multilinear
map, we call $\Phi$ {\it completely compact} if for each $\epsilon
> 0$ there exists a finite dimensional subspace $F\subseteq\Y$
such that
\[\dist(\Phi^{(m)}(x_1,\dots,x_n), M_m(F)) < \epsilon,\] for all
$x_i\in M_{m}(\X_i)$, $\|x_i\|\leq 1$, $i = 1,\dots,n$, and all
$m\in\bb{N}$. We
denote by $\CC(\X_1\timess\X_n, \Y)$ %
the space of all completely bounded completely compact multilinear
maps from $\X_1\timess\X_n$ into $\Y$. A straightforward
verification shows the following:

\begin{remark}\label{cceq}
A completely bounded map $\Phi : \X_1\timess\X_n\rightarrow\Y$ is completely compact if and only if its
linearisation $\tilde{\Phi} : \X_1\haags\X_n\rightarrow\Y$ is completely compact.
\end{remark}

In view of this remark, we frequently identify the spaces
$\CC(\X_1\timess\X_n,\Y)$ and $\CC(\X_1\haags\X_n,\Y)$. The next
result is essentially due to Saar (see Lemmas~1 and~2
of~\cite{saar}).

\begin{proposition}\label{ccc}
(i) \ The space $\CC(\X_1\timess\X_n, \Y)$ is closed in
$\CB(\X_1\timess\X_n, \Y)$.

(ii) Let ${\mathcal E}$, ${\mathcal F}$ and ${\mathcal G}$ be
operator spaces. If $\Phi\in \CC ({\mathcal E},{\mathcal F})$ and
$\Psi\in \CB({\mathcal F},{\mathcal G})$  then $\Psi\circ\Phi\in
\CC({\mathcal E},{\mathcal G})$. If $\Phi\in \CC ({\mathcal
F},{\mathcal G})$ and $\Psi\in \CB({\mathcal E},{\mathcal F})$  then
$\Phi\circ\Psi\in \CC({\mathcal E},{\mathcal G})$.
\end{proposition}

Let $\Hs$ be Hilbert spaces. Recall the  isometry
\[\gamma:\B(H_1)\ehaags \B(H_n)\to \CBS(\B(H_2,H_1)\timess \B(H_n,H_{n-1}),\B(H_n,H_1))\] from
Theorem~\ref{cs}. Let us identify a completely bounded map defined
on $\B(H_2,H_1)\timess \B(H_{n},H_{n-1})$ with the
corresponding completely bounded map defined on
\[\Bh\defeq\B(H_2,H_1)\haags \B(H_{n},H_{n-1}).\] For $u\in \B(H_1)\ehaags \B(H_n)$ we let $\gamma_0(u)$ be the restriction
of $\gamma(u)$ to \[\Kh\defeq\K(H_2,H_1)\haags \K(H_{n},H_{n-1}).\]

\begin{proposition}
  \label{prop:ran-gamma0}
  The map $\gamma_0$ is an isometry from $\B(H_1)\ehaags\B(H_n)$
  onto
  $\CB(\Kh,\B(H_n,H_1))$.
\end{proposition}
\begin{proof}
  Let $\Phi\in\CB(\Kh,\B(H_n,H_1))$.
  Since $\Phi$ is completely bounded, its second
  dual
  $$\Phi^{**} : \B(H_2,H_1)\shaags\B(H_{n},H_{n-1}) \rightarrow \B(H_n,H_1)^{**}$$ is completely bounded
  (here $\shaag$ denotes the normal Haagerup tensor
  product~\cite{effros_ruan}). Let $Q : \B(H_n,H_1)^{**}\rightarrow\B(H_n,H_1)$ be the
  canonical projection. The multilinear map
  \[\tilde{\Phi} : \B(H_2,H_1)\times\dots\times \B(H_{n},H_{n-1})\rightarrow
  \B(H_n,H_1)\] corresponding to $Q\circ \Phi^{**}$ is completely bounded and, by (5.22)
  of~\cite{effros_ruan}, weak* continuous in each variable. By
  Theorem~\ref{cs}, there exists an element $u \in
  \B(H_1)\ehaags\B(H_n)$
  such that $\tilde{\Phi} = \gamma(u)$.
  Hence $\gamma_0(u)=\gamma(u)|_{\Kh}=\tilde\Phi|_{\Kh}=\Phi$.
Thus $\gamma_0$ is surjective.

Fix $u\in \B(H_1)\ehaags \B(H_n)$. From the definition of
$\gamma_0$ we have $\|\gamma_0(u)\|_{\cb}\leq \|\gamma(u)\|_{\cb}
= \|u\|_{\eh}$. On the other hand, the restrictions of the maps
$Q\circ \gamma_0(u)^{**}$ and $\gamma(u)$ to $\Kh$ coincide, and
since both maps are weak* continuous, $\gamma(u) = Q\circ
\gamma_0(u)^{**}|_{\Bh}$. Hence,
$$\|u\|_{\eh} \leq \|Q\circ \gamma_0(u)^{**}\|_{\cb} \leq
\|\gamma_0(u)^{**}\|_{\cb} = \|\gamma_0(u)\|_{\cb}.$$ Thus,
$\gamma_0$ is an isometry.
\end{proof}

\begin{theorem}\label{muccc}
  Let $\Hs$ be Hilbert spaces. The image under $\gamma_0$ of the operator space
  $\E\defeq \K(H_1)\haag(\B(H_2)\ehaags\B(H_{n-1}))\haag \K(H_n)$ is
  $\F\defeq \CC(\Kh,\K(H_n,H_1))$.
\end{theorem}
\begin{proof}
  We first establish the inclusion $\gamma_0(\E)\subseteq\F$. If
  $\Phi=\gamma_0(u)$ where $u\in \E$ then, by Proposition
  \ref{prop:ran-gamma0},
  $\Phi$ is the limit in the
  cb norm of maps of the form $\gamma_0(v)$, where
  \[v = a\odot B\odot b\in \K(H_1)\odot (\B(H_2)\ehaags\B(H_{n-1}))\odot \K(H_n),\] $a$ and $b$ have finite rank and $B$
  is a finite matrix with entries in $\B(H_2)\ehaags\B(H_{n-1})$. But each such map has finite rank and hence is
  completely compact. Moreover, every operator in the image of
    $\gamma_0(v)$ has range contained in the range of~$a$, which is
    finite dimensional. It follows that $\Phi$ takes compact
    values;
    it is completely compact by Proposition~\ref{ccc}.

    To see that $\F\subseteq\gamma_0(\E)$, let $\Phi\in\F$. We
    will assume for technical simplicity that $H_1,\dots,H_n$ are
    separable.
    Let $\{p_{k}\}_{k}$ (resp.~$\{q_{k}\}_{k}$) be a sequence
    of projections of finite rank on $H_1$ (resp.~$H_n$) such
    that $p_{k}\rightarrow I$ (resp.~$q_{k}\rightarrow I$)
    in the strong operator topology. Let $\Psi_{k} : \K(H_n,H_1)\rightarrow \K(H_n,H_1)$ be
    the complete contraction given by $\Psi_k(x) = p_{k}xq_{k}$.

    Let $\epsilon > 0$. Since $\Phi$ is completely compact there
    exists a subspace $F\subseteq \K(H_n,H_1)$ of dimension
    $\ell < \infty$ such that $\dist(\Phi^{(m)}(x), M_m(F)) < \epsilon$
    whenever $x\in  M_m(\Kh)$ has norm at most one. Denote the
    restriction of $\Psi_{k}$ to $F$ by $\Psi_{k,F}$ and let $\iota$
    be the inclusion map
    $\iota:F\hookrightarrow\K(H_n,H_1)$. By~\cite[Corollary~2.2.4]{effros_ruan1},
$\|\Psi_{k, F} - \iota\|_{\cb}\leq \ell \|\Psi_{k,F} -
    \iota\|$. Since $F\subseteq \K(H_n,H_1)$, we have that
    $\Psi_{k,F}(x)\rightarrow x$ in norm for each $x\in F$. It follows
    easily that there exists $k_0$ such that
     $\|\Psi_{k,F} - \iota\|_{\cb} < \epsilon$ whenever $k\geq k_0$.

Let $x\in M_m(\Kh)$ be of norm at most one. Then there exists $y\in
M_m(F)$ such that $\|\Phi^{(m)}(x) - y\| < \epsilon$. Note that
\[\|y\|\leq \|\Phi^{(m)}(x) - y\| + \|\Phi^{(m)}(x)\| \leq \epsilon +
\|\Phi\|_{\cb}.\]
Let $\Phi_{k} = \Psi_{k}\circ\Phi$. If $k\geq k_0$ then
\begin{align*}
\|(\Phi_{k}^{(m)}- \Phi^{(m)})(x)\|  & \leq \|\Phi_{k}^{(m)}(x) -
\Psi_{k}^{(m)}(y)\| + \|\Psi_{k}^{(m)}(y) - y\| +  \|y -
  \Phi^{(m)}(x)\|
  \\ & = \|\Psi_{k}^{(m)}(\Phi^{(m)}(x) - y)\| +
  \|(\Psi_{k,F}-\iota)^{(m)}(y)\| + \|y - \Phi^{(m)}(x)\|
\\ & \leq 2\epsilon + \epsilon (\epsilon + \|\Phi\|_{\cb}).
\end{align*}
This shows that $\|\Phi_{k} - \Phi\|_{\cb}\rightarrow 0$.

By Proposition~\ref{ccc}, it only remains to prove that each
$\Phi_{k}$ lies in $\gamma_0(\E)$.  By
Proposition~\ref{prop:ran-gamma0}, there exists an element \[u =
A_1\odot A_2\odot\dots\odot A_{n-1}\odot A_n\in
\B(H_1)\ehaags\B(H_n)\] where $A_1 : H_1^{\infty}\rightarrow H_1$,
$A_i : H_{i}^{\infty}\rightarrow H_i^{\infty}$, $i = 2,\dots,n-1$
and $A_{n} : H_n\rightarrow H_{n}^{\infty}$ are bounded operators,
such that $\Phi = \gamma_0(u)$. Observe that $\Phi_k =
\gamma_0(u_k)$ where $u_k=(p_kA_1)\odot A_2\odots A_{n-1}\odot
(A_{n}q_k)$. It therefore suffices to show that $u_k\in \E$ for
each $k$. Fix $k$ and let $p=p_k$, $q=q_k$. The operator $pA_1 :
H_1^{\infty}\rightarrow H_1$ has finite dimensional range and is
hence compact. For $i=1,\dots,n$, let $Q_{i,r} :
H_i^{\infty}\rightarrow H_i^{\infty}$ be a projection with block
matrix whose first $r$ diagonal entries are equal to the identity
operator while the rest are zero. Then by compactness,
$(pA_1)Q_{1,r}\rightarrow pA_1$ and $Q_{n,r}(A_{n}q)\rightarrow
A_{n}q$ in norm as $r\to\infty$.  Let $B = A_2\odots A_{n-1}$,
$C_r = (pA_1)Q_{1,r}\odot B\odot Q_{n,r}(A_{n}q)$, $r\in\bb{N}$,
and $C = (pA_1)\odot B\odot (A_{n}q)$. Then
\begin{align*}
   \|C_r - C\|_\eh & \leq
  \|C_r - (pA_1)Q_{1,r}\odot B\odot (A_{n}q)\|_\eh
  \\&\hspace{6em} +  \|(pA_1)Q_{1,r}\odot B\odot (A_{n}q) - C\|_\eh
  \\[6pt] & \leq
  \|(pA_1)Q_{1,r}\|\,\|B\|\,\|Q_{n,r}(A_{n}q) - A_{n}q\|
  \\&\hspace{6em} +  \|(pA_1)Q_{1,r} -
  pA_1\|\,\|B\|\,\|A_{n}q\|.
\end{align*}
It follows that $\|C_r - C\|_\eh\rightarrow0$ as $r\to\infty$. Our
claim will follow if we show that $C_r\in \E$. To this end, it
suffices to show that if $A_1 = [a_1,\dots,a_r,0,\dots]$ and
$A_{n} = [b_1,\dots,b_r,0,\dots]^t$, where $a_i, b_i$ are
operators of finite rank, then $A_1\odot B\odot A_{n}\in \E$.  Let
$A_1$ and $A_{n}$ be as stated and let $B' = (Q_{2,r}A_2)\odot A_3
\odots A_{n-2}\odot (A_{n-1}Q_{n,r})$. Then $A_1\odot B \odot
A_{n} = A_1\odot B'\odot A_{n+1}$ belongs to the algebraic tensor
product $\K(H_1)\odot (\B(H_2)\ehaags\B(H_{n-1}))\odot \K(H_n)$
and hence to $\E=\K(H_1)\haag
(\B(H_2)\ehaags\B(H_{n-1}))\haag \K(H_n)$.
\end{proof}

\begin{romanremarks}\label{rk:ff}
{(i)} It follows from Theorem~\ref{muccc} that if $\Phi :
\Kh\rightarrow \K(H_n,H_1)$ is a mapping of finite rank whose
image consists of finite rank operators then there exist finite
rank projections $p$ and $q$ on $H_1$ and $H_n$, respectively, and
$u\in (p\K(H_1))\haag(\B(H_2)\ehaags\B(H_{n-1}))\haag (\K(H_n)q)$
such that $\Phi = \gamma_0(u)$.

\smallskip

\noindent {(ii)}
  The identity
  $\E_1\haag(\E_2\ehaag\E_3) = (\E_1\haag\E_2)\ehaag\E_3$
  does not hold in general; for an
  example, take $\E_1=\E_3=\B(H)$ and $\E_2=\bb{C}$.
\smallskip

\noindent {(iii)}  For every $\Phi\in CC(\Kh, \K(H_n,H_1))$
there
 exist
  $A_1\in \K(H_1^{J_1},H_1)$, $A_i\in \B(H_i^{J_{i}},H_i^{J_{i-1}})$,
   $i = 2,\dots,n-1$ and $A_{n} \in\K(H_n,H_n^{J_{n-1}})$ such that
 $$\Phi(x_1\otimes\dots\otimes x_{n-1}) = A_1(x_1\otimes
 1)A_2\dots(x_{n-1}\otimes 1)A_{n},$$
whenever $x_i\in \K(H_{i+1},H_i)$, $i = 1,\dots,n-1$.
 Indeed, by
 Proposition~\ref{muccc}, $\Phi(x_1\otimes\dots\otimes x_{n-1}) =
 A_1(x_1\otimes 1)A_2\dots(x_{n-1}\otimes 1)A_{n}$ for some $A_1\odot
 A_2\odots A_n\in \K(H_1)\haag(\B(H_2)\ehaags \B(H_{n-1}))\otimes_{\hh}\K(H_n)$.
 Using an idea of Blecher and Smith~\cite[Theorem~3.1]{blecher_smith}, we can choose
$A_1=[t_j]_{j\in J_1}\in
 M_{J_1,1}(\K(H_1))\subseteq \B(H_1^{J_1},H_1)$ and
 $A_n=[s_i]_{i\in J_{n-1}}\in M_{1,J_{n-1}}(\K(H_n))\subseteq \B(H_n,H_n^{J_{n-1}})$
 such that the sums $\sum_{i}s_is_i^*$ and $\sum_{j}t_j^*t_j$ converge
 uniformly.
 Then $A_1$ is the norm
 limit of $A_1^{\F}=[t_j^{\F}]_{i\in J_{1}}$, where $\F$ is a finite subset of $J_{1}$
and $t_j^{\F}=t_j$ if $j\in \F$ and $t_j^{\F}=0$ otherwise. Therefore $A_1\in \K(H_1^{J_1},H)$.

 Similarly, $A_{n} \in\K(H_n,H_n^{J_n-1})$.

\end{romanremarks}

In the case $n = 2$, Theorem~\ref{muccc} reduces to the following
result which was established by Saar (Satz~6 of~\cite{saar}) using
the fact that every completely compact completely bounded map on
$\K(H_1,H_2)$ is a linear combination of completely compact
completely positive maps.

\begin{corollary}\label{comh}
A completely bounded map  $\Phi : \K(H_1,H_2)\rightarrow
\K(H_1,H_2)$ is completely compact if and only if there exist
an index set $I$ and families $\{a_i\}_{i\in I}\subseteq \K(H_1)$ and
$\{b_i\}_{i\in I}\subseteq \K(H_2)$ such that the series
$\sum_{i\in I} b_ib_i^*$ and $\sum_{i\in I} a_i^*a_i$
converge uniformly and
$$\Phi(x) = \sum_{i\in I} b_i x a_i, \ \ \ x\in \K(H_1,H_2).$$
\end{corollary}

We note in passing that Theorem~\ref{muccc} together with a result
of Effros and Kishimoto~\cite{effros_kishimoto} yields the
following completely isometric identification:

\begin{corollary}\label{identification}
$CC(\K(H_2,H_1))^{**}\simeq
(\K(H_1)\haag\K(H_2))^{**}\simeq CB(\B(H_2,H_1))$.
\end{corollary}

Saar~\cite{saar} constructed an example of a compact map $\Phi :
\K(H)\to\K(H)$ which is not completely compact (see
Section~\ref{sec:cccm} where we give a detailed account of this
construction). We note that a compact completely positive map
$\Phi:\K(H)\to \K(H)$ is automatically completely compact. Indeed,
the Stinespring Theorem implies that there exist an index set $J$
and a row operator $A=[a_i]_{i\in J}\in \B(H^J,H)$ such that
$\Phi(x) = \sum_{i\in J} a_i x a_i^*$, $x\in\K(H)$. The second
dual $\Phi^{**} : \B(H)\rightarrow \B(H)$ of $\Phi$ is a compact
map given by the same formula.
A standard Banach space
argument shows that $\Phi$ takes values in $\K(H)$, and hence
$\Phi^{**}(I)\in \K(H)$.  This means that $AA^*\in\K(H)$ and so $A\in
\K(H^{\infty},H)$ which easily implies that $\Phi$ is completely
compact.

The previous paragraph shows that there exists a compact
completely bounded map on $\K(H)$ which cannot be written as a
linear combination of compact completely positive maps.

We finish this section with a modular version of
Theorem~\ref{muccc}. Given von Neumann algebras $\A_i\subseteq\B(H_i)$, $i = 1,\dots,n$, we let  $\CC_{\!\A_1',\dots,\A_n'\!} (\Kh,
\K(H_n,H_1))$ denote the space of $\A_1',\dots,\A_n'$-modular
completely compact maps from $\Kh$ into $\K(H_n,H_1)$.

\begin{corollary}\label{mcc}
Let $\A_i\subseteq\B(H_i)$, $i = 1,\dots,n$, be von Neumann
algebras. Set $\K'(\A_i) = \K(H_i)\cap \A_i$, for $i=1$ and $i=n$.
Then 
$$\gamma_0\big(\K'(\A_1) \haag(
\A_2\ehaags\A_{n-1})\haag
\K'(\A_n)\big)=\CC_{\!\A_1',\dots,\A_n'\!}(\Kh, \K(H_n,H_1)).$$
\end{corollary}
\begin{proof} By Theorems~\ref{cs} and~\ref{muccc},
the image of $\K'(\A_1)\haag( \A_2\ehaags\A_{n-1})\haag \K'(\A_n)$
under $\gamma_0$ is contained in $\CC_{\!\A_1',\dots,\A_n'}(\Kh,
\K(H_n,H_1))$. For the converse, fix an element
$\Phi\in\CC_{\!\A_1',\dots,\A_n'\!} (\Kh, \K(H_1,H_n))$. By
Theorem~\ref{muccc}, there exists a unique $u\in
\K(H_1)\haag(\B(H_2)\ehaags\B(H_{n-1}))\haag \K(H_n)$ such that
$\gamma_0(u) = \Phi$. By Theorem~\ref{cs}, $u\in \A_1\ehaags\A_n$.
Lemma~\ref{lem:ss} now shows that $u\in \K'(\A_1)\haag(
\A_2\ehaags\A_{n-1})\haag \K'(\A_n)$.
\end{proof}

\section{Complete boundedness of multipliers}
\label{sec:cbmult}

Our aim in this section is to clarify the relationship between
universal operator multipliers and completely bounded maps, extending
results of~\cite{jtt}.
We begin with an observation which will allow us to deal with the
cases of even and odd numbers of variables in the same manner. We
use the notation established in Section~\ref{sec:prelims}.

\begin{proposition}\label{valinc}
Let  $\As$ be C*-algebras and
$\phi\in M(\As)$. Let $\pi_i$ be a
representation of $\A_i$ on a Hilbert space $H_i$, $i =
1,\dots,n$, and $\pi=\pi_1\otimess \pi_n$. The map $S_{\pi(\phi)}$
takes values in $\K(H_1,H_n)$ if $n$ is odd, and in $\K(H_1^\dd, H_n)$ if $n$ is even.
\end{proposition}
\begin{proof}
For even $n$, this is immediate as observed in~\cite{jtt}. Let $n$
be odd. Assume without loss of generality that $\A_i\subseteq
\B(H_i)$ and $\pi_i$ is the identity representation. We call an
element $\zeta\in\Gamma(\Hs)$ thoroughly
elementary if
\[\zeta=\eta_{1,2}^{\dd}\otimes\xi_{2,3}\otimes \dots\otimes
\xi_{n-1,n}\] where all
$\eta_{j,j+1}^{\dd}=\eta_j^{\dd}\otimes\eta_{j+1}^{\dd}$ and
$\xi_{j-1,j}=\xi_{j-1}\otimes\xi_{j}$ are elementary tensors. The
linear span of the thoroughly elementary tensors is dense in the
completion of $\Gamma(\Hs)$ in $\|\cdot\|_{2,\wedge}$.  Moreover, the
linear span of the elementary tensors $\phi = \phi_1\otimess\phi_n$
is dense in $\B(H_1)\otimess\B(H_n)$. By~(\ref{rk4.3}) and since
$S_\phi(\zeta)$ is linear in both $\phi$
and $\zeta$, it suffices to show that $S_\phi(\zeta)$ is compact when
$\phi$ is an elementary tensor and $\zeta$ is a thoroughly elementary
tensor. However, in this case $S_\phi(\zeta)$ has rank at most~$1$,
since for every $\xi_1\in H_1$, \[
S_\phi(\zeta)\xi_1=\nph_n\theta(\xi_{n-1,n})\ldots\nph_2^\dd\theta(\eta_{1,2}^\dd)\nph_1\xi_1=\Big(\prod_{j=1}^{n-1} (
\phi_j\xi_j,\eta_j)\Big) \phi_n\xi_n.\qedhere\]
\end{proof}

We now establish some notation. Let $\As$ be
C*-algebras and $\phi\in \A_1\otimess\A_n$.
Assume that $n$ is even and let $\pi_1,\dots,\pi_n$ be
representations of $\As$ on $\Hs$,
respectively. Set $\pi=\pi_1\otimess\pi_n$. Using the natural
identifications, we consider the map $S_{\pi(\phi)} :
\Gamma(\Hs)\rightarrow H_1\otimes H_n$ as a map (denoted
in the same way)
$$S_{\pi(\phi)} : \C_2(H_1^{\dd},H_2)\odots\C_2(H_{n-1}^{\dd},H_n)\rightarrow\C_2(H_1^{\dd},H_n).$$
We let
$$\Phi_{\pi(\phi)} : \C_2(H_{n-1}^{\dd},H_n)\odots\C_2(H_1^{\dd},H_2)\rightarrow\C_2(H_1^{\dd},H_n)$$
be the map given on elementary tensors by
$$\Phi_{\pi(\phi)}(T_{n-1}\otimess T_1) =
S_{\pi(\phi)}(T_1\otimess T_{n-1}).$$ Note that if $\nph\in
M(\As)$ then $\Phi_{\pi(\phi)}$ is bounded when the domain is
equipped with the Haagerup norm and the range with the operator
norm. In this case, $\Phi_{\pi(\phi)}$ has a unique extension
(which will be denoted in the same way)
$$\Phi_{\pi(\phi)} : \left(\K(H_{n-1}^{\dd},H_n)\haags\K(H_1^{\dd},H_2),\|\cdot\|_{\hh}\right)\rightarrow
\left(\K(H_1^{\dd},H_n),\|\cdot\|_{\op}\right).$$ If $n$ is odd then
the map $\Phi_{\pi(\phi)}$ is defined in a similar way. The map
$\Phi_{\pi(\phi)}$ will be used extensively hereafter.

The main result of this section is Theorem \ref{ecb} where we show
the relation between the complete boundedness of the mappings
$\Phi_{\pi(\varphi)}$ and the property of $\nph$ of being a
multiplier. We will need the following lemma.

\begin{lemma}\label{amplia}
Let $\A_i\subseteq\B(H_i)$ be a C*-algebra, $i = 1,\dots,n$,
$\nph\in \A_1\otimess\A_n$ and $\psi =
(\id^{(k)}\otimess\id^{(k)})(\nph)$. Suppose that $n$ is even. If
$T_i\in M_k(\C_2(H_i^{\dd},H_{i+1}))$ for  even $i$ and $T_i\in
M_k(\C_2(H_i,H_{i+1}^{\dd}))$ for odd $i$ then
$$\Phi_{\nph}^{(k)}(T_{n-1}\odot\dots\odot T_1) =
\Phi_{\psi}(T_{n-1}\otimess T_1),$$ where we identify
$M_k(\C_2(H_i^{\dd},H_{i+1}))$ with $\C_2((H_i^{\dd})^
{(k)},H_{i+1}^{(k)})$ for even $i$, and
$M_k(\C_2(H_i,H_{i+1}^{\dd}))$ with
$\C_2(H_i^{(k)},(H_{i+1}^{\dd})^{(k)})$ for odd $i$. A similar
statement holds for odd $n$.
\end{lemma}
\begin{proof} To simplify notation, we give the proof for $n=2$;
the general proof is similar. If $\phi=a_1\otimes a_2$
is an elementary tensor then $\Phi_\phi(T)=a_2Ta_1^{\dd}$ and it 
is easily checked that the statement holds. By
linearity, it holds for each $\phi\in \A_1\odot\A_2$. Assume that
$\phi\in \A_1\otimes\A_2$ is arbitrary. Let $\{\phi_m\}\subseteq
\A_1\odot\A_2$ be a sequence converging in the operator norm to
$\phi$ and $\psi_m = (\id^{(k)}\otimes \id^{(k)})(\phi_m)$. 
By~(\ref{rk4.3}), $\Phi_{\nph_m}(
T)\rightarrow\Phi_{\nph}(T)$ in the operator norm, for all
$T\in\C_2(H_1^{\dd},H_2)$. This implies that if $S \in
M_k(\C_2(H_1^{\dd},H_2))$, then $\Phi_{\nph_m}^{(k)}(S)\rightarrow
\Phi_{\nph}^{(k)}(S)$ in the operator norm of
$M_k(\C_2(H_1^{\dd},H_2))$. Since $\psi_m\rightarrow\psi$ in the
operator norm, we conclude that $\Phi_{\psi_m}(
S)\rightarrow\Phi_{\psi}(S)$ in the operator norm of
$\C_2((H_1^{\dd})^{(k)},H_2^{(k)})$. It follows that
$\Phi_{\psi}(S) = \Phi_{\nph}^{(k)}(S)$.
\end{proof}

\begin{theorem}\label{ecb}
Let $\As$ be C*-algebras and $\phi\in \A_1\otimess\A_n$. The following are equivalent:

(i) \ $\nph\in M(\As)$;

(ii) if $\pi_i$ is a representation of $\A_i$, $i = 1,\dots,n$,
and $\pi = \pi_1\otimess\pi_n$ then the map
$\Phi_{\pi(\phi)}$ is completely bounded;

(iii) there exist faithful representations $\pi_i$ of $\A_i$,
$i = 1,\dots,n$, such that if $\pi =
\pi_1\otimess\pi_n$ then the map $\Phi_{\pi(\phi)}$ is
completely bounded.

Moreover, if the above conditions hold and $\pi$ is as in (iii)
then $\|\phi\|_{\mm} = \|\Phi_{\pi(\phi)}\|_{\cb}$.
\end{theorem}
\begin{proof} For technical simplicity we take $n=3$.

(i)$\Rightarrow$(ii) Let $\phi$ $\in$ $M(\A_1,\A_2,\A_3)$ and $\pi_i
: \A_i\rightarrow\B(H_i)$ be a representation, $i = 1,2,3$. Then
$\pi(\nph)\in M(\pi_1(\A_1),\pi(\A_2),\pi_3(\A_3))$; thus, it
suffices to assume that $\A_i\subseteq \B(H_i)$ are concrete
C*-algebras and that $\pi_i$ is the identity representation, $i =
1,2,3$.

Fix $k\in\bb{N}$ and let $\psi = (\id^{(k)}\otimes\id^{(k)}\otimes
\id^{(k)})(\phi)$. Since $\nph\in M(\A_1,\A_2,\A_3)$, the map
$$\Phi_{\psi} : \K(H_{2}^{\dd (k)},H_3^{(k)})\odot\K(H_1^{(k)},H_2^{\dd (k)})\rightarrow \K(H_1^{
(k)},H_3^{(k)})$$ is bounded with norm not exceeding
$\|\nph\|_{\mm}$. By Lemma~\ref{amplia},
$\|\Phi_{\phi}^{(k)}\|\leq \|\nph\|_{\mm}$. Since this inequality
holds for every $k\in\bb{N}$, the map $\Phi_{\phi}$ is completely
bounded.

\medskip

(ii)$\Rightarrow$(iii) is trivial.

\medskip

(iii)$\Rightarrow$(i) We may assume that $\A_i\subseteq \B(H_i)$ and
that $\pi_i$ is the identity representation, $i = 1,2,3$. Let
$\lambda$ be a cardinal number,
$\rho_i = \id^{(\lambda)}$ be the ampliation of the identity
representation of multiplicity $\lambda$, $\psi =
(\rho_1\otimes\rho_2\otimes\rho_3)(\phi)$, and $\tilde{H}_i =
H_i^{\lambda}$, $i = 1,2,3$. Fix $\epsilon > 0$ and $\zeta\in
\Gamma(\tilde{H}_1,\tilde{H}_2,\tilde{H}_3)$. Let
$$\tilde{T} = \tilde{T}_{2}\odot \tilde{T}_1\in \C_2(\tilde{H}_{2}^{\dd},\tilde{H}_3)\odot
\C_2(\tilde{H}_{1},\tilde{H}_2^{\dd})$$ be the element canonically
corresponding to $\zeta$. Then there exist $k\in\bb{N}$ and 
canonical projections $P_i$ from $\tilde{H}_i$ onto the direct sum of $k$
copies of $H_i$ such that
if $T_0 = (P_{3}\tilde{T}_{2}(P_{2}^{\dd}\otimes I))\odot
((P_{2}^{\dd}\otimes I)\tilde{T}_{1}P_1)$ and if $\zeta_0$ is the element of
$\Gamma(H_1^{(k)},H_2^{(k)},H_3^{(k)})$ corresponding to $T_0$ then
$\|\zeta - \zeta_0\|_{2,\wedge}\leq \epsilon$.

Set $\psi_0 = (\id^{(k)}\otimes \id^{(k)}\otimes\id^{(k)})(\nph)$.
Arguing as in Lemma~\ref{amplia}, we see that
$\|\Phi_{\psi_0}(T_0)\|_{\op}=\|\Phi_{\psi}(T_0)\|_{\op}$.
Using~(\ref{rk4.3}) and Lemma~\ref{amplia} we obtain
\begin{align*}
\|S_{\psi}(\zeta)\|_{\op}
& \leq \|S_{\psi}(\zeta - \zeta_0)\|_{\op} +
\|S_{\psi}(\zeta_0)\|_{\op} \leq \|S_{\psi}(\zeta - \zeta_0)\|_{\op} +
\|\Phi_{\psi}(T_0)\|_{\op}\\
& \leq \|\psi\|\|\zeta - \zeta_0\|_{2,\wedge} +
\|\Phi_{\psi_0}(T_0)\|_{\op} \leq \epsilon \|\nph\| +
\|\Phi_{\nph}^{(k)}(T_0)\|_{\op}\\
& \leq \epsilon \|\nph\| +
\|\Phi_{\nph}\|_{\cb}\|T_0\|_{\hh}\\
& \leq \epsilon \|\nph\| +
\|\Phi_{\nph}\|_{\cb}\|P_{3}\tilde{T}_{2}(P_{2}^{\dd}\otimes
I)\|_{\op}\|(P_{2}^{\dd}\otimes I)\tilde{T}_{1}P_1\|_{\op}\\
& \leq \epsilon \|\nph\| +
\|\Phi_{\nph}\|_{\cb}\|\tilde{T}_{2}\|_{\op}\|\tilde{T}_{1}\|_{\op}.
\end{align*}
It follows that
$\|\phi\|_{\id^{(\lambda)},\id^{(\lambda)},\id^{(\lambda)}}\leq
\|\Phi_{\phi}\|_{\cb}$.

Now let $\rho_1,\rho_2,\rho_3$ be arbitrary representations of
$\A_1,\A_2,\A_3$, respectively. Then there exists a cardinal
number $\lambda$ such that each of the representations $\rho_i$ is
approximately subordinate to the representation $\id^{(\lambda)}$
(see \cite{voiculescu} for the definition of approximate
subordination and \cite[Theorem 5.1]{hadwin}).  By Theorem~5.1 of~\cite{jtt},
$\|\phi\|_{\rho_1,\rho_2,\rho_3}\leq
\|\phi\|_{\id^{(\lambda)},\id^{(\lambda)},\id^{(\lambda)}}$; now
the previous paragraph implies that
$\|\phi\|_{\rho_1,\rho_2,\rho_3}\leq \|\Phi_{\phi}\|_{\cb}$. It
follows that $\phi\in M(\A_1,\A_2,\A_3)$ and $\|\phi\|_{\mm} \leq
\|\Phi_{\nph}\|_{\cb}$. As the reversed inequality was already
established, we conclude that $\|\phi\|_{\mm} =
\|\Phi_{\nph}\|_{\cb}$.
\end{proof}

\section{The symbol of a universal multiplier}\label{ssy}

Our aim in this section is to generalise the natural
correspondence between a function $\phi\in
\ell^\infty\ehaag\ell^\infty$ and the Schur multiplier $S_\phi$ on
$\B(\ell^2(\bN))$ given by $S_{\nph}((a_{ij})) =
(\nph(i,j)a_{ij})$. To each universal operator multiplier we will
associate an element of an extended Haagerup tensor product which
we call its symbol. This will be used in the subsequent sections
to identify certain classes of operator multipliers.

Recall that if $\A$ is a C*-algebra, its opposite C*-algebra
$\A^o$ is defined to be the C*-algebra whose underlying set, 
norm, involution and
linear structure coincide with those of $\A$ and whose
multiplication $\cdot$ is given by $a\cdot b = ba$. If $a\in\A$
we denote by $a^o$ the element of $\A^o$ corresponding to $a$.
 If $\pi : \A\rightarrow \B(H)$ is a representation
of $\A$ then the map $\pi^{\dd} : a^o\rightarrow \pi(a)^{\dd}$
from $\A^o$ into $\B(H^{\dd})$ is a representation of $\A^o$.
Clearly, $\pi$ is faithful if and only if $\pi^{\dd}$ is faithful.
If $\pi_i : \A_i\rightarrow\B(H_i)$ are faithful representations,
$i = 1,\dots,n$ ($n$ even), then by~\cite[Lemma~5.4]{effros_ruan}
there exists a complete isometry
$\pi_n\ehaag\pi_{n-1}^{\dd}\ehaags\pi_1^{\dd}$ from
$\A_n\ehaag\A_{n-1}^o\ehaags\A_1^o$ into
$\B(H_n)\ehaag\B(H_{n-1}^{\dd})\ehaags\B(H_1^{\dd})$ which sends
$a_n\otimes a_{n-1}^o\otimess a_1^o$ to $\pi_n(a_n)\otimes
\pi_{n-1}(a_{n-1})^{\dd}\otimess \pi_1(a_1)^{\dd}.$

Henceforth, we will consistently write $\pi=\pi_1\otimess\pi_n$ and
\[\pi'=\begin{cases}
  \pi_n\ehaag\pi_{n-1}^\dd\ehaags \pi_1^\dd\quad&\text{if $n$ is
    even,}\\
  \pi_n\ehaag\pi_{n-1}^\dd\ehaags \pi_1 &\text{if $n$ is
    odd.}
\end{cases}
\]

Let $n\in\bb{N}$, $\As$ be C*-algebras, $\pi_i$ be a
representation of $\A_i$, $i = 1,\dots,n$, and $\phi$ $\in$
$M(\As)$. Assume that $n$ is even. By Theorem~\ref{ecb}, the map
$$\Phi_{\pi(\phi)} : \K(H_{n-1}^{\dd},H_n)\haags\K(H_1^{\dd},H_2)\rightarrow\K(H_1^{\dd},H_n)$$
is completely bounded. By Proposition~\ref{prop:ran-gamma0}, there
exists a unique element $u^{\pi}_{\phi} \in
\B(H_n)\ehaag\B(H_{n-1}^{\dd})\ehaags\B(H_1^{\dd})$ such that
$\gamma_0(u^\pi_\phi)=\Phi_{\pi(\phi)}$. For example, if each $\A_i$ is a
concrete C*-algebra and $a_i\in \A_i$, $i = 1,\dots,n$, then
\[u^{\id}_{a_1\otimes a_2\otimess a_{n-1}\otimes a_n}=a_n\otimes
a_{n-1}^\dd\otimess a_2\otimes a_1^\dd.\] If $n$ is odd then we
define $u^{\pi}_{\phi}$ similarly.

The main result of this section is the following.

\begin{theorem}\label{symb}
Let $\As$ be C*-algebras and $\phi\in M(\As)$. There exists a unique
element
\[u_{\phi}\in
\begin{cases}
\A_n\ehaag\A_{n-1}^{o}\ehaags\A_2\ehaag\A_1^{o}\;&\text{if $n$ is even,}\\
\A_n\ehaag\A_{n-1}^{o}\ehaags\A_2^o\ehaag\A_1&\text{if $n$ is odd}
\end{cases}\]
with the property that if $\pi_i$ is a representation of $\A_i$ for
$i = 1,\dots,n$ then
\begin{equation}\label{usy}
u_{\phi}^{\pi} = \pi'(u_{\phi}).
\end{equation}
The map $\nph\rightarrow u_{\nph}$ is linear and if $a_i\in\A_i$, $i
= 1,\dots,n$ then
\[u_{a_1\otimess a_n} =
\begin{cases}
a_n\otimes a_{n-1}^o\otimess a_1^o\;&\text{if $n$ is even,}\\
a_n\otimes a_{n-1}^o\otimess a_1&\text{if $n$ is odd.}
\end{cases}\]
Moreover, $\|\nph\|_{\mm} = \|u_{\nph}\|_{\eh}$.
\end{theorem}

\begin{definition}\label{dsym}
The element $u_{\phi}$ defined in Theorem~\ref{symb} will be called
the {\it symbol} of the universal multiplier $\phi$.
\end{definition}

In order to prove the theorem we have to establish a number of
auxiliary results.

 If $\omega\in \B(H)^*$ we let
$\tilde{\omega} \in \B(H^{\dd})^*$ be the functional given by
$\tilde{\omega}(a^{\dd}) = \omega(a)$. Note that if $\omega =
\omega_{\xi,\eta}$ is the vector functional $a\mapsto (a\xi,\eta)$
then $\tilde{\omega} = \omega_{\eta^{\dd},\xi^{\dd}}$.

\begin{lemma}\label{sli}
  Let $\A_i\subseteq\B(H_i)$ be a C*-algebra, $\xi_i,\eta_i\in
  H_i$ and $\omega_i = \omega_{\xi_i,\eta_i}$, $i = 1,\dots,n$.
  Suppose that $\phi\in M(\As)$. Then
\begin{equation}\label{key}
(\phi(\xi_1\otimess\xi_n),\eta_1\otimess\eta_n) =
\begin{cases}
  \langle
  u^{\id}_{\phi},\omega_n\otimes\tilde{\omega}_{n-1}\otimess\tilde{\omega}_1\rangle \quad\text{$n$ even,}\\
  \langle
  u^{\id}_{\phi},\omega_n\otimes\tilde{\omega}_{n-1}\otimess{\omega}_1\rangle \quad\text{$n$ odd.}
\end{cases}
\end{equation}
\end{lemma}
\begin{proof}
We only consider the case $n$ is even since the proof for odd
$n$ is similar. Suppose that $\phi$ is an elementary tensor, say
$\phi = a_1\otimess a_n$. Then $u^{\id}_{\phi} = a_n\otimes
a_{n-1}^{\dd}\otimess a_1^{\dd}$ and thus
$$(\phi(\xi_1\otimess\xi_n),\eta_1\otimess\eta_n)
= \prod_{i=1}^n (a_i\xi_i,\eta_i) = \langle
u_{\phi}^{\id},\omega_n\otimes\tilde{\omega}_{n-1}\otimess\tilde{\omega}_1\rangle.$$
By linearity, (\ref{key}) holds for each $\phi\in \A_1\odots\A_n$.

Now let $\phi$ be an arbitrary element of $M(\As)$. By
Theorem~\ref{jtt}, there exists a net $\{\phi_{\nu}\}\subseteq
\A_1\odots\A_n$ and representations $u^{\id}_{\phi} = A_n\odots A_1$
and $u^{\id}_{\phi_{\nu}} = A_n^{\nu}\odots A_1^{\nu}$, where
$A_i^{\nu}$ are finite matrices with entries in $\A_i$ if $i$ is even
and in $\A_i^{\dd}$ if $i$ is odd, such that
$\phi_{\nu}\rightarrow\phi$ semi-weakly, $A_i^{\nu}\rightarrow A_i$
strongly and all norms $\|A_i\|, \|A_i^{\nu}\|$ are bounded by a
constant depending only on $n$. As in~(\ref{product}), we have
\begin{equation}\label{pro}
\langle
u^{\id}_{\phi},\omega_n\otimes\tilde{\omega}_{n-1}\otimess\tilde{\omega}_1\rangle
 =
\langle A_n,\omega_n\rangle\langle
A_{n-1},\tilde{\omega}_{n-1}\rangle\dots\langle
A_1,\tilde{\omega}_1\rangle.
\end{equation}
Moreover,  all norms $\|\langle A_i^{\nu},\omega_i\rangle\|$ (for
even $i$) and $\|\langle A_i^{\nu},\tilde{\omega}_i\rangle\|$ (for
odd $i$) are bounded by a constant depending only on $n$, and the
strong convergence of $A_i^{\nu}$ to $A_i$ implies that $\langle
A_i^{\nu},\omega_i\rangle$ converges strongly to $\langle
A_i,\omega_i\rangle$. Indeed, it is easy to check that if
$\xi,\eta\in H$, $A\in M_I(\B(H))\equiv \B(H\otimes\ell_2(I))$
and $\zeta\in\ell_2(I)$ for some index set $I$ then
$$\|\langle A,\omega_{\xi,\eta}\rangle\zeta\|^2 = \left(
A(\xi\otimes\zeta),\eta\otimes \langle A,\omega_{\xi,\eta}\rangle
\zeta\right).$$ This implies that $\|\langle
A_i-A_i^{\nu},\omega_i\rangle\eta\|\leq
C\|(A_i-A_i^{\nu})(\xi_i\otimes\eta)\|$ for some constant $C > 0$,
and the strong convergence follows.

Since operator multiplication is jointly strongly continuous on
bounded sets, it now follows from~(\ref{pro}) that
$$\langle
u^\id_{\phi_{\nu}},\omega_n\otimes\tilde{\omega}_{n-1}\otimess\tilde{\omega}_1\rangle
 \rightarrow
\langle
u^\id_{\phi},\omega_n\otimes\tilde{\omega}_{n-1}\otimess\tilde{\omega}_1\rangle.$$
On the other hand, since $\phi_{\nu}\rightarrow\phi$ semi-weakly,
$$(\phi_{\nu}(\xi_1\otimess\xi_n),\eta_1\otimess\eta_n)
\rightarrow
(\phi(\xi_1\otimess\xi_n),\eta_1\otimess\eta_n).$$
The proof is complete. \end{proof}

\begin{lemma}
  \label{lem:eh-uniq}
  Let $H_i$ be a Hilbert space and $\E_i\subseteq \B(H_i)$
  be an operator space, $i=1,\dots,n$. %
  Suppose that $\X$ and $\Y$ are closed
  subspaces of $\E_1$ and $\E_n$, respectively and let $u,v\in \E_1\haags\E_n$.
  If
  \[ R_\omega(u)\in \X\qand L_{\omega'}(v)\in \Y\]
  whenever $\omega=\omega_2\otimess\omega_n$ and
  $\omega'=\omega_1'\otimess\omega_{n-1}'$
  where every $\omega_i$, $\omega_i'\in \B(H_i)_*$ is a vector functional, then
  \[u\in \X\ehaag\E_2\ehaags\E_n\qand v\in \E_1\ehaags\E_{n-1}\ehaag\Y.\]
\end{lemma}
\begin{proof}
 Let $\F_i$ be the span of the vector functionals on
  $\B(H_i)$. By linearity, $R_\omega(u)\in \X$ for each
  $\omega\in \F_2\odots\F_n$. Now suppose that
  \[\omega\in (\B(H_2)\ehaags\B(H_n))_*=\C_1(H_2)\haags\C_n(H_n).\]
  There exists a sequence $(\omega_m)\subseteq \F_2\odots\F_n$ such
  that $\omega_m\to \omega$ in norm. Hence
  \[ \|R_\omega(u)-R_{\omega_m}(u)\|_{\B(H_1)} \leq
  \|\omega-\omega_m\|\,\|u\|_{\eh}\to 0,\]
  whence $R_\omega(u)=\lim_m R_{\omega_m}(u)\in \X$. Spronk's
  formula~(\ref{eq:spronk}) now implies that $u\in
  \X\ehaag\E_2\ehaags\E_n$. The assertion concerning $v$ has a similar
  proof.
\end{proof}

We will use slice maps defined on the minimal tensor product of
several C*-algebras as follows. Assume that $\A_i\subseteq\B(H_i)$
and $\omega_i\in \B(H_i)^*$, $i = 1,\dots,n$, and let $\phi\in
\A_1\otimess\A_n$. If $1\leq i_1 < \dots < i_k\leq n$ and
$\{\ell_1,\dots,\ell_{n-k}\}$ is the complement of $\{i_1,\dots,i_k\}$
in $\{1,\dots,n\}$, let
$$\Lambda_{\omega_{i_1},\dots,\omega_{i_k}} : \A_1\otimess\A_n\rightarrow\A_{\ell_1}\otimess\A_{\ell_{n-k}}$$
be the unique norm continuous linear mapping given on elementary
tensors by
$$\Lambda_{\omega_{i_1},\dots,\omega_{i_k}}(a_1\otimess a_n) =
\omega_{i_1}(a_{i_1})\dots\omega_{i_k}(a_{i_k})\,a_{\ell_1}\otimess
a_{\ell_{n-k}}.$$

\begin{proposition}\label{symin}
Let $\A_i\subseteq\B(H_i)$, $i = 1,\dots,n$, be C*-algebras
and let $\phi\in M(\As)$. Then
\[u_{\phi}^{\id}\in
\begin{cases}
\A_n\ehaag\A_{n-1}^{\dd}\ehaags\A_1^{\dd}\;&\text{if $n$ is even,}\\
\A_n\ehaag\A_{n-1}^{\dd}\ehaags\A_1&\text{if $n$ is odd.}
\end{cases}\]
\end{proposition}
\begin{proof}
We only consider the case $n=3$. Let $u = u^{\id}_{\phi}$;
by definition, $u\in\B(H_3)\ehaag\B(H_{2}^{\dd})\ehaag\B(H_1)$. Let $\xi_i,\eta_i\in H_i$ and $\omega_i =
\omega_{\xi_i,\eta_i}$, $i = 1,2,3$. Then by~(\ref{eq:slices})
and Lemma~\ref{sli},
\begin{align*}
(R_{\tilde{\omega}_{2}\otimes\omega_1}(u)\xi_3,\eta_3)
& = \langle
R_{\tilde{\omega}_{2}\otimes{\omega}_1}(u),
\omega_{3}\rangle
 =
\langle u,\omega_3\otimes\tilde{\omega}_{2}\otimes{\omega}_1\rangle\\
& =
(\phi(\xi_1\otimes\xi_2\otimes\xi_3),\eta_1\otimes\eta_2\otimes\eta_3)
 = (\Lambda_{\omega_1,\omega_{2}}(\phi)\xi_3,\eta_3).
\end{align*}
Thus
$$R_{\tilde{\omega}_{2}\otimes\tilde{\omega}_1}(u)
= \Lambda_{\omega_1,\omega_{2}}(\phi)\in\A_3.$$
Lemma~\ref{lem:eh-uniq} now implies that $u\in\A_3\ehaag\B(H_{2}^{\dd})\ehaag\B(H_1)$.

Let $w = R_{{\omega}_1}(u)$.  By the
previous paragraph, $w\in\A_3\ehaag\B(H_{2}^{\dd})$.
By~(\ref{eq:slices}) and Lemma~\ref{sli},
\begin{align*}
(L_{\omega_3}(w)\eta_{2}^\dd,\xi_{2}^\dd)&= \langle
L_{\omega_3}(w), \tilde{\omega}_{2} \rangle
=
\langle R_{\omega_{1}}(u),
\omega_3\otimes\tilde \omega_{2}\rangle
\\ &=
\langle u, \omega_3\otimes\tilde\omega_{2}\otimes\omega_1\rangle
=
(\Lambda_{\omega_1,\omega_3}(\phi)\xi_{2},\eta_{2})
=
(\Lambda_{\omega_1,\omega_3}(\phi)^\dd
\eta_{2}^\dd,\xi_{2}^\dd).
\end{align*}
Hence $L_{\omega_3}(w) =
\Lambda_{\omega_1,\omega_3}(\phi)^\dd\in\A_{2}^{\dd}$ and, by
Lemma~\ref{lem:eh-uniq}, $w\in \A_3\ehaag\A_{2}^{\dd}$. Applying
this lemma again shows that $u\in
\A_3\ehaag\A_{2}^{\dd}\ehaag\B(H_1).$ Continuing in this fashion
we see that $u\in \A_3\ehaag\A_{2}^{\dd}\ehaag\A_1$. \end{proof}

\begin{lemma}
  \label{lem:compsymb}
  Let $\As$ be C*-algebras and let \[\rho_i:\A_i\to \B(K_i),\qquad
  \theta_i:\rho_i(\A_i)\to \B(H_i)\] be representations,
  $i=1,\dots,n$. Suppose that

  (i) for any cardinal number~$\kappa$, the
  representations $\theta_i^{(\kappa)}:\rho_i(\A_i)\to \B(H_i^\kappa)$
  are strongly continuous, and

  (ii) whenever $\phi\in M(\As)$
  and $\{\phi_\nu\}$ is
  a net in $\A_1\odots\A_n$ such that $\rho(\phi_\nu)\to \rho(\phi)$
  semi-weakly and $\sup_{\nu}\|\phi_\nu\|_m < \infty$ then $\Phi_{\theta\circ\rho(\phi_\nu)}\to
  \Phi_{\theta\circ \rho(\phi)}$ pointwise weakly.

Then $u_\phi^{\theta\circ\rho}=\theta'(u_\phi^\rho)$, for each
$\phi\in M(\As)$.
\end{lemma}
\begin{proof}
  If $\phi=a_1\otimess a_n$ is an elementary tensor, then
$u_\phi^\rho=\rho'(a_n\otimes a_{n-1}^o\otimess a_1^o)$, so
$$u_\phi^{\theta\circ\rho}=(\theta\circ\rho)'(a_n\otimes
      a_{n-1}^o\otimess a_1^o)
      =\theta'(u_\phi^\rho).$$
By linearity, the claim also holds for $\phi\in \A_1\odots\A_n$.

  If $\phi\in M(\As)$ is arbitrary then $\rho(\phi)\in
  M(\rho(\A_1),\dots,\rho(\A_n))$ and by Theorem~\ref{jtt} and
  Proposition~\ref{symin}, there exist a net
  $\{\phi_{\nu}\}\subseteq \A_1\odots\A_n$ such that
  $\rho(\phi_{\nu})\rightarrow\rho(\phi)$ semi-weakly, a
  representation $u_{\phi}^\rho = A_n\odots A_1$, where $A_i\in
  M_\kappa(\rho_i(\A_i))\subseteq\B(K_i^\kappa)$ if $i$ is even and
  $A_i\in M_\kappa(\rho_i^\dd(\A_i^o))\subseteq \B(K_i^\kappa)^\dd$ if
  $i$ is odd ($\kappa$ being a suitable index set), whose operator
  matrix entries belong to $\rho_i(\A_i)$ if $i$ is even and to
  $\rho_i^\dd(\A_i^o)$ if $i$ is odd, and representations
  $u_{\phi_{\nu}}^\rho = A_n^{\nu}\odots A_1^{\nu}$ where the
  $A_i^{\nu}$ are finite matrices with operator entries in
  $\rho_i(\A_i)$ if $i$ is even and $\rho_i^\dd(\A_i^o)$ if $i$ is odd
  such that $A_i^{\nu}\rightarrow A_i$ strongly and all norms
  $\|A_i^{\nu}\|$, $\|A_i\|$ are bounded.

  Now $\theta'(u_{\phi}^\rho) = \tilde{A}_n\odots\tilde{A}_1$ and
  $\theta'(u_{\phi_{\nu}}^\rho) = \tilde{A}_n^{\nu}\odots\tilde{A}_1^{\nu}$
  where $\tilde A_i$ and $\tilde A_i^\nu$ are the images of $A_i$ and
  $A_i^\nu$ under
  $\theta_i^{(\kappa)}$ or $(\theta_i^\dd)^{(\kappa)}$
  according to whether $i$ is even or odd.
  By assumption~(i),
\begin{equation}\label{someq}
  \gamma_0(\theta'(u_{\phi_{\nu}}^\rho))(T_{n-1}\otimess
  T_1)\rightarrow \gamma_0(\theta'(u_{\phi}^\rho))(T_{n-1}\otimess
  T_1)
\end{equation}
  weakly for all $T_{n-1}\in
  \C_2(H_{n-1}^{\dd},H_n)$,$\dots$,$T_1\in\C_2(H_1^{\dd},H_2)$. On the other hand,
  assumption~(ii) and the first paragraph of the proof show that
  \[\gamma_0(\theta'(u_{\phi_{\nu}}^\rho))
  =\gamma_0(u_{\phi_\nu}^{\theta\circ \rho})=
  \Phi_{\theta\circ \rho(\phi_{\nu})}\to \Phi_{\theta\circ \rho(\phi)}=
  \gamma_0(u_\phi^{\theta\circ \rho})
  \]
  pointwise weakly. Using (\ref{someq}) we conclude that
$\gamma_0(u_\phi^{\theta\circ \rho}) =
\gamma_0(\theta'(u_{\phi}^\rho))$; since $\gamma_0$ is injective
we have that $u_\phi^{\theta\circ \rho} = \theta'(u_{\phi}^\rho)$.
\end{proof}

\begin{proof}[Proof of Theorem~\ref{symb}] We will only consider
the case $n$ is even. Let $\rho_i : \A_i\rightarrow\B(K_i)$ be the
universal representation of $\A_i$, $i = 1,\dots,n$. Set
$\rho=\rho_1\otimess \rho_n$ and
$\rho'=\rho_n\otimes\rho_{n-1}^\dd\otimess\rho_1^\dd$. By
Proposition~\ref{symin}, $u_\phi^\rho$ lies in the image of
$\rho'$; we define $u_{\phi} = (\rho')^{-1}(u^{\rho}_{\phi})$.

Let $\kappa$ be a nonzero cardinal number and let
$\sigma_i=\rho_i^{(\kappa)}$. If
$\theta_i=\id_{\rho_i(\A_i)}^{(\kappa)}=\sigma_i\circ\rho_i^{-1}$
then it follows from the proof of Proposition~6.2 of~\cite{jtt}
that the hypotheses of Lemma~\ref{lem:compsymb} are satisfied, so
\[ u_\phi^\sigma=u_\phi^{\theta\circ \rho}
=\theta'(u_\phi^\rho)=(\theta'\circ \rho')(u_\phi)=\sigma'(u_\phi).\]

Now let $\pi_i$ be an arbitrary representation of $\A_i$. It is
well known (see e.g.~\cite{tak1}) that $\pi_i$ is unitarily
equivalent to a subrepresentation of $\sigma_i=\rho_i^{(\kappa)}$
for some~$\kappa$. Hence there exist unitary operators $v_i$, $i =
1,\dots,n$ (acting between appropriate Hilbert spaces) and
subspaces $H_i$ of $K_i^\kappa$, such that if $\tau_i(x) = v_i x
v_i^*|_{H_i}$ then $\pi_i = \tau_i\circ \sigma_i$.  Examining the
proof of Proposition~6.2 of~\cite{jtt}, we see that $\tau$
satisfies the hypotheses of Lemma~\ref{lem:compsymb}, so
\[
u_\phi^\pi=u_\phi^{\tau\circ \sigma}=\tau'(u_\phi^\sigma)=(\tau\circ\sigma)'(u_\phi)=\pi'(u_\phi).
\]

The uniqueness of $u_{\nph}$ follows from the injectivity of
$\gamma_0$. The linearity of the map $\nph\rightarrow u_{\nph}$
and its values on elementary tensors are straightforward. The fact
that $\|\nph\|_{\mm} = \|u_{\nph}\|_{\eh}$ follows from
Proposition~\ref{prop:ran-gamma0} and Theorem~\ref{ecb}.
\end{proof}

\begin{nonumremarks}
(i) Let $\A_i\subseteq\B(H_i)$, $i = 1,\dots,n$ be
    concrete C*-algebras of operators. Taking $\pi_i$ to be the
    identity representation for $i=1,\dots,n$ and writing
    $\id=\pi_1\otimess\pi_n$ gives
    $u_\phi=u^\id_\phi$
    if we identify $\A_i^o$ with $\A_i^\dd$.
    \smallskip

\noindent (ii)  Theorem~\ref{symb} implies that if
$\A_i$, $i = 1,\dots,n$, are concrete C*-algebras then
the entries of the block operator matrices $A_i$
appearing in the representation of $\nph$ in
Theorem~\ref{jtt} can be chosen from 
$\A_i$, $i = 1,\dots,n$. 
\end{nonumremarks}

\section{Completely compact multipliers}
\label{schar}

In this section we introduce the class of completely compact
multipliers and characterise them within the class of all
universal multipliers using the notion of the symbol introduced in
Section \ref{ssy}. We will need the following lemma.

\begin{lemma}\label{eqm}
  Let $\A_i\subseteq \B(H_i)$ be a C*-algebra, $i=1,\dots,n$,
  $a\in\A_1$, $b\in \A_n$ and $\phi\in M(\As)$. Let $\psi\in \A_1\otimess\A_n$ be given by
  \[\psi =
  \begin{cases}
    (a\otimes I\otimess I\otimes b)\phi\qquad\qquad&\text{if $n$ is even,}\\
    (a\otimes I\otimess I\otimes I)\phi(I\otimess I\otimes b)&\text{if $n$ is odd.}
  \end{cases}\] Then $\psi\in M(\As)$ and
  \begin{equation}\label{for} \Phi_\psi(x) =
  \begin{cases}
    b\Phi_\phi(x)a^\dd\quad&\text{if $n$ is even,}\\
    b\Phi_\phi(x)a &\text{if $n$ is odd.}
  \end{cases}
  \end{equation}
\end{lemma}
\begin{proof} For technical simplicity, we will only consider the case $n=2$.
Let $a_i\in \A_i$, $i = 1,2$, and $\phi = a_1\otimes a_2$. In this
case $\psi = (aa_1)\otimes (ba_2)$ so 
$$\Phi_{\psi}(T) =
ba_2T(aa_1)^\dd=ba_2Ta_1^{\dd}a^{\dd}=b\Phi_{\phi}(T)a^\dd.
$$
By linearity, (\ref{for}) holds whenever $\phi\in \A_1\odot \A_2$.

Assume that $\phi\in M(\A_1,\A_2)$ is arbitrary. Fix an operator
$T\in\C_2(H_1^{\dd},H_2)$. By Theorem~\ref{jtt}, there exists a net
$\{\phi_{\nu}\}\subseteq\A_1\odot\A_2$ such that
$\phi_{\nu}\rightarrow\phi$ semi-weakly,
$\sup_{\nu}\|\nph_{\nu}\|_{\mm} < \infty$ and
$\Phi_{\phi_{\nu}}( T)\rightarrow
\Phi_{\phi}(T)$ weakly.

Let $\psi_{\nu} = (a\otimes b)\phi_{\nu}$; then
$\psi_{\nu}\rightarrow\psi$ semi-weakly. Clearly, $\psi_\nu\in
\A_1\odot\A_2$; in particular $\psi_\nu\in M(\A_1,\A_2)$. By
the previous paragraph, $\Phi_{\psi_{\nu}}(\cdot) =
b\Phi_{\phi_{\nu}}(\cdot)a^{\dd}$ and hence
$\Phi_{\psi_{\nu}}( T)\rightarrow
b\Phi_{\phi}(T)a^{\dd}$ weakly. If $\phi_\nu =B_1^{\nu}\odot B_2^{\nu}$
then $\psi_{\nu} = (a B_1^{\nu})\odot ((b\otimes I)B_2^{\nu})$. 
It follows from Theorem~\ref{jtt} that $\psi\in M(\A_1,\A_2)$ and that
$\Phi_{\psi_{\nu}}(T)\rightarrow
\Phi_{\psi}(T)$ weakly. Thus
$\Phi_{\psi}(T) = b\Phi_{\phi}(T)a^{\dd}$.
\end{proof}

Given faithful representations $\pi_1,\dots,\pi_n$ of the
C*-algebras $\As$, respectively, we define
  \begin{align*}
    M^\pi_{cc}(\As)&=\{\phi\in M(\As):\Phi_{\pi(\phi)} \text{ is completely compact}\}
    \\[6pt]
    M^\pi_{\ff}(\As)&=\{\phi\in M(\As): \text{the range of $\Phi_{\pi(\phi)}$}\\
    &\qquad\text{is a finite dimensional space of finite-rank operators}\}.
  \end{align*}

\begin{theorem}\label{charpi}
  Let $\A_i\subseteq\B(H_i)$ be a C*-algebra, $i = 1,\dots,n$, and $\phi\in M(\As)$. The
  following are equivalent: \smallskip

  \noindent
  (i) \ \ $\phi\in M_{cc}^\id (\As)$;\smallskip

  \noindent
  (ii) \
\[u_{\phi}^{\id}\in
\begin{cases}
(\K(H_n)\cap \A_n) \haag(\A_{n-1}^{\dd}\ehaags\A_2)\haag (\K(H_1^\dd)\cap \A_1^{\dd})\;&\text{$n$ even,}\\
(\K(H_n)\cap \A_n) \haag(\A_{n-1}^{\dd}\ehaags\A_2^{\dd})\haag
(\K(H_1)\cap \A_1)&\text{$n$ odd;}
\end{cases}\]

  \noindent
  (iii) there exists a net $\{\phi_{\alpha}\}\subseteq M_{\ff}^\id(\As)$ such that
  $\|\phi_{\alpha} - \phi\|_{\mm}\rightarrow 0$.
\end{theorem}
\begin{proof}
  We will only consider the case $n$ is even.

\smallskip

  (i)$\Rightarrow$(ii)
  Theorem~\ref{muccc} implies that
  \[u_\phi^\id\in \K(H_n)\haag (\B(H_{n-1}^{\dd})\ehaags\B(H_2))\haag
  \K(H_1^{\dd})\] while, by Proposition~\ref{symin}, \[u_{\phi}^\id
  \in
  \A_n\ehaag\A_{n-1}^{\dd}\ehaags\A_2\ehaag\A_1^{\dd}.\]
  The conclusion now follows from Lemma~\ref{lem:ss}.

\smallskip

  (ii)$\Rightarrow$(i) By Theorem~\ref{muccc},
  $\Phi_{\phi}=\gamma_0(u_\phi^\id)$ is completely compact.

\smallskip

  (ii)$\Rightarrow$(iii) Let $p\in\B(H_1)$ (resp.~$q\in\B(H_n)$) be
  the projection onto the span of all ranges of operators in $\K(H_1)\cap\A_1$ (resp.~$\K(H_n)\cap\A_n$), and let
  $\{p_{\alpha}\}\subseteq\K(H_1)\cap\A_1$
  (resp.~$\{q_{\alpha}\}\subseteq\K(H_n)\cap\A_n$) be a net of
  finite rank projections which tends strongly to $p$ (resp.~$q$). It
  is easy to see that $\Phi_{\phi}(T_{n-1}\otimess T_1) =
  q\Phi_{\phi}(T_{n-1}\otimess T_1)p^{\dd}$, for all $T_1\in\K(H_1^{\dd},H_2), \dots, T_{n-1}\in \K(H_{n-1}^{\dd},H_n)$. Let
  $\phi_{\alpha} = (p_{\alpha}\otimes I\otimess I\otimes
  q_{\alpha})\phi$. By
  Lemma~\ref{eqm}, $\phi_{\alpha}\in M(\As)$ and
  $\Phi_{\nph_{\alpha}}(\cdot) =
  q_{\alpha}\Phi_{\phi}(\cdot)p_{\alpha}^{\dd}$; hence $\phi_\alpha\in
  M_{\ff}^\id(\As)$. We have already seen that $\Phi_\phi$ is completely
  compact, and it follows from the proof of
  Theorem~\ref{muccc} that $\Phi_{\nph_{\alpha}}\rightarrow\Phi_\phi$ in the
  cb norm. By Theorem~\ref{ecb}, $\|\phi -
  \phi_{\alpha}\|_{\mm}\rightarrow 0$.

\smallskip

  (iii)$\Rightarrow$(i)
  is immediate from Proposition~\ref{ccc}
  and Theorem~\ref{ecb} and the fact that finite rank maps are completely
  compact.
\end{proof}

Now consider the sets
\begin{align*}
  M_{cc}(\As)&=\bigcup_\pi M_{cc}^\pi(\As)\\
  M_{\ff}(\As)&=\bigcup_\pi M_{\ff}^\pi(\As)
\end{align*}
where the unions are taken over all $\pi = \opis$, each $\pi_i$
being a faithful representation of $\A_i$. We refer to the first
of these as the set of completely compact multipliers.

\begin{lemma}
  \label{lem:ff}
  If $\rho_i$ is the reduced atomic representation of $\A_i$,
  $i=1,\dots,n$, and $\rho=\rho_1\otimess\rho_n$ then  $M_{\ff}(\As)= M_{\ff}^\rho(\As)$.
\end{lemma}
\begin{proof}
  Again, we give the proof for the even case only.
  We must show that $M_{\ff}^\pi(\As)\subseteq M_{\ff}^\rho(\As)$
  whenever $\pi = \pi_1\otimess\pi_n$ where each $\pi_i$ is
  a faithful representation of $\A_i$. Without loss of generality, we
  may assume that each $\pi_i$ is the identity representation of
  $\A_i\subseteq \B(H_i)$. Let $\nph\in M^\pi_{\ff}(\As)$ so that the
  range of $\Phi_{\nph}$ is finite dimensional and consists
  of finite rank operators. By Remark~\ref{rk:ff}~(i)
  there exist finite rank projections $p$ and $q$ on
  $H_1^{\dd}$ and $H_n$, respectively, such that $u_{\nph}^{\id}$ lies
  in the intersection of
$$(q\K(H_n))\haag(\B(H_{n-1}^{\dd})\ehaags\B(H_{2}))\haag (\K(H_1^{\dd})p)$$ and $\A_n\ehaags\A_1^{\dd}$. By Lemma~\ref{lem:ss}, $u_{\nph}^{\id}$ lies in
$$(q\K(H_n)\cap \A_n)\haag(\B(H_{n-1}^{\dd})\ehaags\B(H_{2}))\haag (\K(H_1^{\dd})p\cap\A_1^{\dd}).$$ Hence there
exists a representation $u_{\nph}^{\id} = A_n\odot\dots\odot A_1$
of $u_{\nph}^{\id}$ such that $A_n = qA_n$ and $A_1 = A_1 p$.
Suppose that $A_n = [b_1,b_2,\dots]$, where $b_j\in \A_n$ for
each $j$, and let $q_j$ be the orthogonal projection onto the range of
$b_j$. Setting $Q_m = \bigvee_{j=1}^m
q_j$ we see that $\{Q_m\}$ is an increasing sequence of
projections in~$\A_n$ dominated by $q$. It follows that
$\bigvee_{m=1}^{\infty} Q_m\in \A_n$. We may thus assume that
$q\in\A_n$. Similarly, we may assume that $p\in\A_1^{\dd}$. Now
$$\rho'(u_{\nph}) = (\rho_n(q)\rho_n(A_n))\odot\dots\odot
(\rho_1(A_1)\rho_1(p)).$$ By~\cite{ylinen}, $\rho_n(q)$ and
$\rho_1(p)$ have finite rank. By Lemma~\ref{eqm}, $\nph\in
M^\rho_{\ff}(\As)$.
\end{proof}

We are now ready to prove the main result of this section.

\begin{theorem}\label{charcn}
  Let $\As$ be C*-algebras
  and $\phi\in M(\As)$. The following are
  equivalent:\smallskip

(i) \ \ $\phi\in M_{cc}(\As)$;\smallskip

(ii) \ $u_{\phi}\in
\begin{cases}
\K(\A_n) \haag(\A_{n-1}^o\ehaags\A_2)\haag \K(\A_1^o)\;&\text{if $n$ is even,}\\
\K(\A_n) \haag(\A_{n-1}^o\ehaags\A_2^o)\haag \K(\A_1)&\text{if $n$ is odd;}
\end{cases}$\smallskip

(iii) there exists a net $\{\phi_{\alpha}\}\subseteq M_{\ff}(\As)$ such that $\|\phi_{\alpha} -
\phi\|_{\mm}\rightarrow 0$.
\end{theorem}
\begin{proof}
  We will only consider the case $n$ is even.

\smallskip

  (i)$\Rightarrow$(ii) Choose $\pi=\pi_1\otimess\pi_n$ such that
  $\phi\in M^\pi_{cc}(\As)$; after identifying $\A_i$
  with its image under $\pi_i$, we may assume that each $\pi_i$ is the identity
  representation of a concrete C*-algebra $\A_i\subseteq\B(H_i)$.
  By Theorem~\ref{charpi}, $u_\phi^\id$ lies in
  \[(\K(H_n)\cap \A_n) \haag(
  \A_{n-1}^o\ehaags\A_2)\haag (\K(H_1^\dd)\cap \A_1^o).\]
  The conclusion follows from the fact that
  $\K(H_i)\cap \A_i \subseteq \K(\A_i)$ for $i=1,n$.

\smallskip

  (ii)$\Rightarrow$(i)
  Let $\rho_i$ be the reduced atomic representation $\A_i\to\B(H_i)$ for
  $i=1,\dots,n$. Since $\rho'$ is an isometry,
  $u^\rho_\phi=\rho'(u_\phi)$ lies in
  \[
      \rho_n(\K(\A_n)) \haag(
    \rho_{n-1}^\dd(\A_{n-1}^o)\ehaags\rho_2(\A_2))\haag \rho_1^\dd(
    \K(\A_1^o)).
  \]
  By Theorem~7.5 of~\cite{ylinen68},
  $\K(H_i)\cap\rho_i(\A_i)=\rho_i(\K(\A_i))$.
  By Theorem~\ref{charpi}, $\phi\in M_{cc}^\rho(\As)$.

\smallskip

(i)$\Rightarrow$(iii) is immediate from Theorem~\ref{charpi}.

\smallskip

(iii)$\Rightarrow$(i) Suppose that $\{\phi_{\alpha}\}\subseteq
M_{\ff}(\As)$ is a net such that $\|\phi_{\alpha} -
\phi\|_{\mm}\rightarrow 0$. By Lemma~\ref{lem:ff},
$\{\phi_{\alpha}\}\subseteq M_{\ff}^{\rho}(\As)$, where $\rho$ is
the tensor product of the reduced atomic representations of $
\A_1,\dots,\A_n$. By Theorem~\ref{charpi}, $\nph\in
M^\rho_{cc}(\As)\subseteq M_{cc}(\As)$.
\end{proof}

In the next theorem we show that in the case $n = 2$ one more
equivalent condition can be added to those of
Theorem~\ref{charcn}.

\begin{theorem}\label{charctwo}
Let $\A$ and $\B$ be C*-algebras and $\phi\in M(\A,\B)$.
The following are equivalent:

(i) \ $\phi\in M_{cc}(\A,\B)$;

(ii) there is a sequence $\{\phi_k\}_{k=1}^{\infty}\subseteq
\K(\A)\odot \K(\B)$ such that $\|\phi_k - \phi\|_{\mm}
\rightarrow 0$ as $k\to \infty$.
\end{theorem}
\begin{proof} (i)$\Rightarrow$(ii) By Theorem~\ref{charcn},
$u_{\nph}\in \K(\B)\haag\K(\A^o)$; thus $u_{\nph} =
\sum_{i=1}^{\infty}b_i\otimes a_i^o$ where $a_i^o\in \K(\A^o)$,
$b_i\in \K(\B)$, $i\in\bb{N}$, and the series
$\sum_{i=1}^{\infty}b_ib_i^*$ and $\sum_{i=1}^{\infty}a_i^{o
*}a_i^o$ converge in norm. Let $\nph_k = \sum_{i=1}^k a_i\otimes
b_i\in\A\odot\B$. By Theorem~\ref{symb}, $u_{\nph_k} =
\sum_{i=1}^k b_i\otimes a_i^o$ and $\|\nph - \nph_k\|_{\mm} =
\|u_{\nph} - u_{\nph_k}\|_{\eh}\rightarrow0$ as $k\to\infty$.
\medskip

(ii)$\Rightarrow$(i) Assume that $\A$ and $\B$ are represented
concretely. It is clear that $\nph_k\in M_{cc}(\A,\B)$. By
Theorem~\ref{ecb}, $\|\Phi_{\id(\nph)} -
\Phi_{\id(\nph_k)}\|_{\cb} = \|\nph - \nph_k\|_{\mm}$. Proposition~\ref{ccc} now implies that $\Phi_{\id(\nph)}$ is completely
compact, in other words, $\nph\in M_{cc}(\A,\B)$.
\end{proof}

\section{Compact multipliers}\label{sec:cccm}

In this section we compare the set of completely compact
multipliers with that of compact multipliers. We exhibit
sufficient conditions for these two sets of multipliers to
coincide, and show that in general they are distinct. Finally, we
address the question of when any universal multiplier in the
minimal tensor product of two C*-algebras is automatically
compact. We show that this happens precisely when one of the
C*-algebras is finite dimensional while the other coincides with
the set of its compact elements.

\subsection{Automatic complete compactness}

We will need the following result complementing Theorem~\ref{muccc}.
Notation is as in Section~\ref{completelyc}.

\begin{proposition}\label{ocm}
If $\Phi : \Kh\rightarrow \K(H_n,H_1)$ is a compact completely
bounded map then $\gamma_0^{-1}(\Phi)\in \K(H_1)\ehaag\B(H_2)\ehaags\B(H_{n-1})\ehaag\K(H_n)$.
\end{proposition}
\begin{proof}
  Fix $\varepsilon>0$. By compactness, there exist $y_1,\ldots, y_\ell\in \K(H_n,H_1)$
  such that
  $\min_{1\leq i\leq \ell}\|\Phi(x)-y_i\|<\varepsilon$
  for each $x\in\Kh$ with $\|x\|\leq 1$.

Let $\{p_{\alpha}\}$ (resp.~$\{q_{\alpha}\}$) be a net of
finite rank projections finite rank in $\K(H_1)$ (resp.~$\K(H_n)$) such
that $p_{\alpha}\to I$ (resp.~$q_{\alpha}\to I$) strongly and let
$\Phi_{\alpha} : \Kh\rightarrow \K(H_n,H_1)$ be the map given by
$\Phi_{\alpha}(x) = p_{\alpha}\Phi(x)q_{\alpha}$. Let $u =
\gamma_0^{-1}(\Phi)$ and $u_{\alpha} =
\gamma_0^{-1}(\Phi_{\alpha})$. Since each $y_i$ is compact there
exists $\alpha_0$ such that
$\|p_{\alpha}y_iq_{\alpha}-y_i\|<\varepsilon$ for $i=1,\ldots,\ell$
and $\alpha\geq\alpha_0$. Moreover, for any $x\in\Kh$, $\|x\|\leq
1$ and $\alpha\geq \alpha_0$, we have
\begin{eqnarray*}
\|\Phi_{\alpha}(x)-\Phi(x)\| & \leq & \min_{1\le i\le
\ell}\{\|\Phi_{\alpha}(x)-p_{\alpha}y_iq_{\alpha}\|
+\|p_{\alpha}y_iq_{\alpha}-y_i\|+ \|y_i-\Phi(x)\|\}\\
& \leq &
\min_{1\le i\le
\ell}\{2\|\Phi(x)-y_i\|+\|p_{\alpha}y_iq_{\alpha}-y_i\|\}\leq
3\varepsilon,
\end{eqnarray*}
so $\|\Phi_{\alpha}-\Phi\|\to 0$. Remark~\ref{rk:ff} (i) shows
that $u_{\alpha}\in
\K(H_1)\haag(\B(H_2)\ehaags\B(H_{n-1}))\haag\K(H_n)$;
it follows that for every $\omega\in
(\B(H_2)\ehaags\B(H_{n-1})\ehaag\B(H_n))_*$ we have
$R_{\omega}(u_{\alpha})\in \K(H_1)$.

Suppose that $\xi_i,\eta_i\in H_i$ and let
$\omega_i=\omega_{\xi_i,\eta_i}$ be the corresponding vector
functional.
Lemma~\ref{sli} and a straightforward verification shows that if
$v\in \B(H_1)\ehaags\B(H_n)$ has a representation of the form $v =
A_1\odot\ldots\odot A_n$ and $\omega = \omega_2\otimess\omega_n$
then
\begin{equation}\label{mixx}
  (R_{\omega}(v)\xi_1,\eta_1) = \langle
  v,\omega_1\otimes\ldots\otimes\omega_n\rangle =  (\gamma_0(v)(\zeta)\xi_n,\eta_1),
\end{equation}
where
\[\zeta=
\big((\eta_2^*\otimes\xi_1)\otimes(\eta_3^*\otimes\xi_2)\otimess
(\eta_{n-1}^*\otimes\xi_{n-2})\otimes(\eta_n^*\otimes\xi_{n-1})\big)\in \Kh
\]
is an elementary tensor whose components are rank one operators.

Since $\gamma_0(u_{\alpha}) \rightarrow \gamma_0(u)$ in norm,
(\ref{mixx}) implies that $R_{\omega}(u_{\alpha})\to
R_{\omega}(u)$ in the operator norm of $\K(H_1)$. Since
$R_{\omega}(u_{\alpha})\in \K(H_1)$, we obtain $R_{\omega}(u)\in
\K(H_1)$. By Lemma~\ref{lem:eh-uniq}, $u\in 
\K(H_1)\ehaag\B(H_2)\ehaags\B(H_n)$. Similarly we see that
$u\in \B(H_1)\ehaag\B(H_2)\ehaags\B(K_n)$; the
conclusion now follows.
\end{proof}

\begin{nonumremark}
  The converse of Proposition~\ref{ocm} does not hold, even for $n=2$. Indeed, let
  $\{p_i\}_{i=1}^\infty$ be a family of pairwise orthogonal rank one
  projections on a Hilbert space~$H$ and let $u = \sum_{i=1}^{\infty}
  p_i\otimes p_i$. Then $u\in \K(H)\ehaag\K(H)$ and the range of
  $\gamma_0(u)$ consists of compact operators, but
  $\gamma_0(u)(p_i)=p_i$ for each $i$, so $\gamma_0(u)$ is not
  compact.
\end{nonumremark}

Given C*-algebras $\As$, we let $M_{c}(\As)$ be the collection of
all $\nph\in M(\As)$ for which there exist faithful
representations $\pi_1,\dots,\pi_n$ of $\As$, respectively, such
that if $\pi = \pi_1\otimess\pi_n$ then the map $\Phi_{\pi(\phi)}$
is compact. We call the elements of $M_{c}(\As)$ {\it compact
multipliers}.

As a consequence of the previous result we obtain the following
fact.

\begin{proposition}\label{ehaag}
Let $\As$  be C*-algebras and let $\nph\in M_c(\As)$. Then
\[u_{\phi}\in
\begin{cases}
\K(\A_n) \ehaag\A_{n-1}^o\ehaags\A_2\ehaag \K(\A_1^o)\;&\text{if $n$ is even,}\\
\K(\A_n) \ehaag\A_{n-1}^o\ehaags\A_2^o\ehaag \K(\A_1)&\text{if $n$ is odd.}
\end{cases}
\]
\end{proposition}
\begin{proof}
We only consider the case $n$ is even. We may assume that
$\A_i\subseteq \B(H_i)$ is a concrete non-degenerate C*-algebra,
$i = 1,\dots,n$, and that $\Phi_{\nph}$ is compact. By
Propositions~\ref{symin} and~\ref{ocm}, $u_{\nph}^{\id}$ belongs
to
$$\big(\K(H_n)\ehaag\B(H_{n-1}^{\dd})\ehaags\B(H_{2})\ehaag\K(H_1^{\dd})\big)
\cap \big(\A_n\ehaags\A_1^{\dd}\big).$$ Since $\K(H_n)\cap
\A_n\subseteq\K(\A_n)$ and $\K(H_1^{\dd})\cap
\A_1^{\dd}\subseteq\K(\A_1^\dd)$, an application of
(\ref{eq:spronk}) shows that $u_{\nph}^{\id}\in \K(\A_n)
\ehaag\A_{n-1}^{\dd}\ehaags\A_2\ehaag \K(\A_1^{\dd})$.
\end{proof}

If $\{\A_j\}_{j\in J}$ is a family of C*-algebras, we will denote by
$\oplus^{c_0}_{j\in J} \A_j$ and~$\oplus^{\ell_{\infty}}_{j\in J} \A_j$ their 
$c_0$- and $\ell_{\infty}$-direct sums, respectively.

\begin{theorem}\label{otherc}
  Let $\A_1, \dots,\A_n$ be C*-algebras, and
suppose 
  that $\K(\A_1)$ is isomorphic to $\bigoplus^{c_0}_{j\in J}
  M_{m_j}$ and $\K(\A_n)$ is isomorphic to $\bigoplus^{c_0}_{j\in J}
  M_{n_j}$ where $J$ is some index set and $\sup_{j\in J}m_j$ and
  $\sup_{j\in J} n_j$ are finite.  Then \[M_c(\As)=
  M_{cc}(\As).\]
\end{theorem}
\begin{proof}
  We give the proof for $n=3$; the case of a general $n$ is similar.
Let $m=\sup\{m_j,n_j:j\in J\}$. By hypothesis, $\K(\A_1)$ and
  $\K(\A_3)$ may both be embedded in the C*-algebra
  $\C\defeq\bigoplus^{c_0}_{j\in J}M_m$ for some $m\in\bb{N}$;
  without loss of generality, we
  may assume that this embedding is an inclusion and that $\A_i$ 
  is represented faithfully on some Hilbert space $H_i$ such
  that $H_1$ and $H_3$ both contain the Hilbert space
  
  $H=\bigoplus_{j\in J} \bC^m$.
  Given $\phi\in
  M_c(\A_1,\A_2,\A_3)$, Proposition~\ref{ehaag} implies that the symbol
  $u_{\phi}$ of $\phi$ can be written in the form $u_\phi = A_3\odot
  A_{2}\odot A_1$, where the entries of $A_3$ and
  $A_1$ belong to $\C$.  Let $\{e_{ij}:i,j=1,\dots,m\}$
  be the canonical matrix unit system of $M_m$ and let $P_k =
  \bigoplus_{j\in J} e_{kk} \in \bigoplus^{\ell^{\infty}}_{j\in J} M_m$,
  $k=1,\dots,m$. For $k,\ell,s,t=1,\dots,m$, we set
  $A_3^{k,\ell}=P_kA_3(P_\ell\otimes I)$ and
  $A_1^{s,t}=(P_s\otimes I)A_1P_t$ and define
  \[ u_{k,\ell,s,t}=A_3^{k,\ell}\odot
  A_2\odot A_1^{s,t}  \qand \Phi_{k,\ell,s,t}=\gamma_0(u_{k,\ell,s,t}).\]
  Then
  $\gamma_0(u_{\varphi})=\Phi = \sum_{k,\ell,s,t} \Phi_{k,\ell,s,t}$ so it
  suffices to show that each of the maps $\Phi_{k,\ell,s,t}$ is
  completely compact. Now
  \begin{multline*}
    \Phi_{k,\ell,s,t}(T_{2}\otimes T_1) =
    P_k\Phi(P_\ell T_{2}\otimes T_1P_s)P_t
    =A_3^{k,\ell}((P_\ell T_{2})\otimes I)A_{2}
    ((T_1P_s)\otimes I)A_1^{s,t}.
  \end{multline*}
  Thus, $\Phi_{k,\ell,s,t}$ can be considered as a completely bounded
  multilinear map from
  $\K(H_{2}^{\dd},P_\ell H)\times\K(P_s H,H_2^\dd)$ into $\K(P_t
  H,P_kH)$.
   Since $\Phi$ is compact,
  it follows that $\Phi_{k,\ell,s,t}$ is compact.

  Take a basis $\{e_i^j:i=1,\ldots,m, j\in J\}$ of
  $H=\bigoplus_{j\in J} {\mathbb C}^m$, where for each $j\in J$,
  the standard basis of the $j$-th copy
  of ${\mathbb C}^m$ is $\{e_i^j: i=1,\ldots, m\}$. 
  Let $U_{k}:P_k H\to P_1H$ be the unitary
  operator defined by $U_{k}e_{k}^j=e_1^j$. Consider the
  mapping $\Psi: \K(H_{2}^{\dd},P_1H)\times\K(P_1 H,H_2^\dd)\to\K(P_1H,P_1H)$ given by
  \[
  \Psi(T_{2}\otimes T_1)=U_{k}\Phi_{k,\ell,s,t}(U_{\ell}T_{2}\otimes
  T_1U_{s})U_{t}.
  \]
  To show that $\Phi_{k,\ell,s,t}$ is completely compact it
  suffices to show that $\Psi$ is.
  Let $\C_0 = P_1\C P_1$; then
  $\C_0$ is isomorphic to $c_0$ and its commutant $\C_0'$
  has a cyclic vector. Moreover, $\Psi$ is a $\C_0'$-modular
  multilinear map.
   Let $\{p_{\alpha}\}$ be a net of finite dimensional projections
   belonging to $\C_0$, such that s-$\lim p_{\alpha}=I_{P_1 H}$.
   Consider the completely bounded multilinear maps $\Psi_{\alpha}(x)
   = p_{\alpha}\Psi(x)p_{\alpha}$.  Since the ranges of the $p_{\alpha}$'s
   are finite dimensional, $\Psi_\alpha$ has finite rank
   and is hence completely compact.
   Since $\Psi$ is compact, we may argue as in the proof of
   Proposition~\ref{ocm} to show that $\|\Psi_{\alpha}-\Psi\|\to 0$.
   Now the maps $\Psi$ and $\Psi_{\alpha}$ are $\C_0'$-modular
   and $\C_0'$ has a cyclic vector, so by the generalisation~\cite[Lemma~3.3]{jtt} of a 
   result of Smith~\cite[Theorem~2.1]{smith},
   $$\|\Psi_{\alpha}-\Psi\|_{\cb}=\|\Psi_{\alpha}-\Psi\|\to 0.$$
   Proposition~\ref{ccc} now implies that $\Psi$ is completely compact.
\end{proof}

The following corollary extends Proposition 5 of~\cite{hladnik} to
the case of multidimensional Schur multipliers. We recall
from~\cite{jtt} that with every $\nph\in
\ell_{\infty}(X_1\times\dots\times X_n)$ we associate a mapping
$S_{\nph} : \ell_2(X_1\times
X_2)\odot\dots\odot\ell_2(X_{n-1}\times X_n) \rightarrow
\ell_2(X_1\times X_n)$ which extends the usual Schur
multiplication in the case $n = 2$. We equip the domain of
$S_{\nph}$ with the Haagerup norm where each of the terms is given
its operator space structure arising from its embedding into the
corresponding space of Hilbert-Schmidt operators endowed with the
operator norm.

\begin{corollary}\label{cmultd}
Let $X_1,\dots,X_n$ be sets and $\nph\in
\ell_{\infty}(X_1\times\dots\times X_n)$. The following are
equivalent:

(i) $S_{\nph}$ is compact;

(ii) $\nph\in
c_0(X_1)\haag(\ell_{\infty}(X_2)\ehaags\ell_{\infty}(X_{n-1}))\haag c_0(X_n)$.
\end{corollary}

\begin{proof}
Assume first that $S_{\nph}$ is compact. For notational simplicity
we assume that $X_i={\mathbb N}$, $i=1,\ldots,n$. It follows from \cite[Section~3]{jtt}
that the map $S_\nph$ induces a completely bounded compact map
$$\hat S_{\nph}:\C_2\times\ldots \times\C_2\to \C_2$$
defined by $\hat
S_\nph(T_{f_1},\ldots,T_{f_n})=T_{S_{\nph}(f_1,\ldots,f_n)}$,
where $T_f$ is the Hilbert-Schmidt operator with kernel $f$. By
Proposition~\ref{ocm}, $\varphi=\gamma_0^{-1}(\hat S_{\varphi})\in
\K(\ell_2)\ehaag
\B(\ell_2)\ehaag\ldots\ehaag\B(\ell_2)\ehaag\K(\ell_2)$. Since
$S_{\varphi}$ is bounded, $\nph$ is a Schur multiplier and
by~\cite[Theorem~3.4]{jtt}, $\nph\in
\ell_{\infty}\ehaag\ldots\ehaag \ell_{\infty}$. Hence $\nph\in
c_0\ehaag \ell_{\infty}\ehaag\ldots\ehaag \ell_{\infty}\ehaag
c_0$. We may now argue as in the last paragraph of 
the preceding proof to
show that $\nph\in c_0\haag (\ell_{\infty}\ehaags
\ell_{\infty})\haag c_0$.
\end{proof}

Our next aim is to show that if both $\K({\mathcal A_1})$ and
$\K(\A_n)$ contain full matrix algebras of arbitrarily large sizes
then the completely compact multipliers form a proper subset of
the set of compact multipliers. Saar~\cite{saar} has provided an
example of a compact completely bounded map on $\K(H)$ (where $H$
is a separable Hilbert space) which is not completely compact. It
turns out that Saar's example also shows that the sets of compact
and completely compact multipliers are distinct, in the case under
consideration.

We will need some preliminary results. Let ${\mathcal A}$ and
${\mathcal B}$ be $C^*$-algebras. Recall that a linear map
$\Phi:{\mathcal A}\to{\mathcal B}$ is called symmetric (or
hermitian) if $\Phi = \Phi^*$ where $\Phi^*:{\mathcal
A}\to{\mathcal B}$ is the map given by $\Phi^*(a)=(\Phi(a^*))^*$.
By $S_{\mathcal A}$ we denote the unit ball of ${\mathcal A}$ and
set $S_{\mathcal A}^h=\{a\in S_{\mathcal A}: a=a^*\}$. The
following lemma is a special case of Satz~6 of~\cite{saar}. We
include a direct proof for the convenience of the reader.

\begin{lemma}\label{lemma-appendix}
Let $H$ be a Hilbert space. If $\Phi:{\mathcal A}\to \K(H)$ is a
symmetric, completely compact linear map with $\|\Phi\|_{\cb}\leq
1$, then there exists a positive operator $c\in\K(H)$ such that
$\Phi^{(n)}(a)\leq c\otimes 1_n$ for all $a\in S_{M_n({\mathcal
A})}^h$ and all $n\in\bb{N}$. Moreover, $c$ can be chosen to have
norm arbitrarily close to one.
\end{lemma}
\begin{proof}
We first show that for a given $\varepsilon > 0$ there exists a
finite rank projection $p$ on $H$ such that
\begin{equation}\label{phin}
\|\Phi^{(n)}(a)-(p\otimes 1_n)\Phi^{(n)}(a)(p\otimes 1_n)\|\leq
\varepsilon \ \text{ for any } a\in S_{{M_n(\A)}}.
\end{equation}
Since $\Phi$ is completely compact, there exists a finite
dimensional subspace $F\subset \K(H)$ such that
$\text{dist}(\Phi^{(n)}(a),M_n(F))\leq\varepsilon/3$ for any
$a\in{{M_n(\A)}}$, $\|a\|\leq 1$ and any $n\in\bb{N}$. Let $S_{F,
1+\varepsilon}=\{x\in F:\|x\|\leq 1+\varepsilon\}$ and let $k=\dim
F$.
Choose a finite rank projection~$p\in\K(H)$ such
that
\[\|x-pxp\|<\frac\epsilon{k(3+\epsilon)}\qtext{for all $x\in S_{F,1+\epsilon}$}\]
and let $\Psi:F\to \K(H)$ be defined by $\Psi(x)=x-pxp$.
By~\cite[Corollary~2.2.4]{effros_ruan1}, $\Psi$ is completely
bounded and $\|\Psi\|_{\cb}\leq k\|\Psi\|$. This implies that
\begin{equation*}\label{lemma19}
\|\Psi^{(n)}(y)\|\leq
k\|\Psi\|\,\|y\|\leq \frac{\varepsilon}{3+\epsilon}\|y\|\leq\frac\epsilon3
\end{equation*}
for all $y\in M_n(F)$ with $\|y\|\leq 1+\epsilon/3$.

Now for $a\in S_{M_n({\mathcal A})}^h$ let  $y\in M_n(F)$ be such
that $\|\Phi^{(n)}(a)-y\|\leq\varepsilon/3.$ Then $\|y\|\leq
\|\Phi^{(n)}(a)\|+\varepsilon/3\leq 1+\varepsilon/3$. Hence
\begin{align*}
&\|\Phi^{(n)}(a)-(p\otimes 1_n)\Phi^{(n)}(a)(p\otimes 1_n)\|\\
&\quad  \leq \|\Phi^{(n)}(a)-y\|+\|\Psi^{(n)}(y)\|+\|(p\otimes
1_n)(y-\Phi^{(n)}(a))(p\otimes 1_n)\|\\ &\quad \leq
\varepsilon/3+\varepsilon/3+\varepsilon/3=\varepsilon,
\end{align*}
proving~(\ref{phin}). Next we fix $\varepsilon>0$ and choose a
finite rank projection $q_1$ on $H$ such that
\begin{equation*}\label{Phin}
\|\Phi^{(n)}(a)-(q_1\otimes 1_n)\Phi^{(n)}(a)(q_1\otimes
1_n)\|\leq\frac{\varepsilon}{2}, \quad a\in M_n({\mathcal A}),\
\|a\|\leq 1,\  n\in\bb{N}.
\end{equation*}
Let $r_1 : \A\rightarrow\K(H)$ be the mapping given by
$r_1(a)=\Phi(a)-q_1\Phi(a)q_1$, $a\in\A$. Then
$r_1=\Psi\circ\Phi$, where $\Psi:\K(H)\to\K(H)$ is
the completely bounded map given by $\Psi(x)=x-q_1xq_1$. By
Proposition~\ref{ccc}, $r_1$ is completely compact. Moreover,
$\displaystyle \|r_1\|_{\cb}\leq \varepsilon/2$ and
$\Phi(a)=q_1\Phi(a)q_1+r_1(a)$, $a\in{\mathcal A}$. Proceeding by
induction, we can find sequences of finite rank projections $q_i$
and completely compact symmetric mappings $r_i$ such that
$\displaystyle \|r_i\|_{\cb}\leq \varepsilon/2^i$ and
$$\Phi(a)=q_1\Phi(a)q_1+\sum_{i=1}^{\infty}q_{i+1}r_i(a)q_{i+1},
\ \ a\in{\mathcal A}.$$ Let $c=
q_1+\sum_{i=1}^{\infty}\frac{\varepsilon}{2^i}q_{i+1}.$ We have
that $\Phi^{(n)}$ and $r_i^{(n)}$ are symmetric and
$$\Phi^{(n)}(a)=(q_1\otimes 1_n)\Phi^{(n)}(a)(q_1\otimes 1_n)+
\sum_{i=1}^{\infty}(q_{i+1}\otimes 1_n)r_i^{(n)}(a)(q_{i+1}\otimes
1_n),$$ for each $a\in{\mathcal A}$. Now
$$\Phi^{(n)}(a)\leq (q_1\otimes 1_n)\|\Phi\|_{\cb}+\sum_{i=1}^{\infty}
(q_{i+1}\otimes 1_n)\|r_i\|_{\cb}\leq
(q_1+\sum_{i=1}^{\infty}\frac{\varepsilon}{2^i}q_{i+1})\otimes
1_n=c\otimes 1_n$$ for all $a\in S_{M_n({\mathcal A})}^h$. By
construction, $c$ is compact and $\|c\|\leq 1+\varepsilon$.
\end{proof}

Let $H$ be an infinite dimensional separable Hilbert space and
$\{q_k\}_{k\in{\mathbb N}}$ be a family of pairwise orthogonal
projections in $\B(H)$ with $\text{rank }q_k = k$ and
$\sum_{k=1}^{\infty}q_k = I$. Set $p_n=\sum_{k=1}^{n} q_k$,
$n\in\bb{N}$. Let $\Phi_k : \B(q_k H)\to \B(q_k H)$, $k\in{\mathbb
N}$, be symmetric linear maps such that
\begin{equation}\label{cond}
\|\Phi_k\|_{\cb}=1,\quad \|\Phi_k\|\rightarrow 0\ \text{as $k\to\infty$,} \
\ \mbox{ and } \ \ \sum_{k=1}^{\infty}\|\Phi_k\|_2^2 < \infty,
\end{equation}
where $\|\Phi_k\|_2$ denotes the norm of the mapping
$\Phi_k$ when $\B(q_k H)\simeq \C_2(q_k H)$ is equipped with
the Hilbert-Schmidt norm.
Identifying $\B(q_kH)$ with $q_k\B(H)q_k$, let $\Phi : \K(H)\to \B(H)$
be the map given by the norm-convergent sum
\begin{equation}\label{exfi}
\Phi(x)=\sum_{k=1}^\infty\!\strut^\oplus\, \Phi_k(q_kxq_k),\quad x\in \K(H).
\end{equation}

An example of such a map is obtained by taking
$\Phi_k=k^{-1}\tau_k$ where $\tau_k$ is the transposition map
$\B(q_kH)\simeq M_k\to M_k\simeq \B(q_kH)$, which is symmetric and
an isometry for both the operator and the Hilbert-Schmidt norm. It
is well known (see e.g.~\cite[p.~419]{pisier-book}) that
$\|\tau_k\|_{\cb}= k$ and hence conditions (\ref{cond}) are
satisfied.

The next lemma is a straightforward extension of~\cite[pp.~32--34]{saar}.

\begin{lemma}\label{saar}
If $\Phi$ is a map satisfying~(\ref{cond}) and~(\ref{exfi}) 
then the range of $\Phi$ consists of compact operators. Moreover, $\Phi$ is
completely contractive and compact but not completely compact.
\end{lemma}
\begin{proof}  Fix $x\in\K(H)$.
Since $\|\Phi_k\|\rightarrow_{k\to \infty} 0$ as we have
$p_n\Phi(x)p_n\rightarrow \Phi(x)$ in norm, so $\Phi(x)\in\K(H)$.
Each of the maps $x\mapsto \Phi_k(q_kxq_k)$ is completely
contractive, so $\Phi$ is completely contractive.

Next, note that $\Phi$ maps the unit ball of $\K(H)$ into
$U\defeq U_1\oplus U_2\oplus\cdots$, where $U_k$ is the
ball of radius $\|\Phi_k\|$ in $q_k\B(H)q_k$. Since $U$ is compact,
the map $\Phi$ is compact.

If $\Phi$ were completely compact then by Lemma~\ref{lemma-appendix},
there would exist a positive compact operator $c$ on $H$ such that
\begin{equation*}\label{cocomp1}
\Phi^{(k)}(x)\leq c\otimes 1_k\text{ for all } x\in
S_{M_k(\K(H))}^h \text{ and all }k\in{\mathbb N}.
\end{equation*}
Hence for every $k\in\bN$ and  $x\in S_{M_k(\B(H))}^h$,
\begin{equation*}\label{cocomp2}
\Phi_k^{(k)}((q_k\otimes 1_k)x(q_k\otimes
1_k)) = (q_k\otimes 1_k)\Phi^{(k)}(x)(q_k\otimes
1_k) \leq q_k c q_k\otimes 1_k.
\end{equation*}
However, $\|\Phi_k^{(k)}\|=\|\Phi_k\|_{\cb}=1$ by~\cite{smith83} and
$\Phi_k^{(k)}$ is symmetric, so
$$\|q_k c q_k\|=\|q_k c q_k\otimes 1_k\|\geq\sup\{\|\Phi_k^{(k)}(x)\|:{x\in
  S^h_{M_k(q_k\B(H)q_k)}}\}= 1,$$
which is impossible since~$c$ is compact.
\end{proof}

\begin{lemma}\label{HS}
Let $\C = \bigoplus^{c_0}_{k\in \bb{N}}\B(q_k H) \subseteq \K(H)$.
Then there exists $\nph\in M(\C^{\dd},\C)$ such that $\Phi =
\Phi_{\id(\nph)}$.
\end{lemma}
\begin{proof}
Let $\nph_k\in \B(q_k H)^{\dd}\otimes\B(q_k H)$ be such that
$\Phi_{\id(\nph_k)} = \Phi_k$, $k\in\bb{N}$, where the family
$\{\Phi_k\}_{k=1}^{\infty}$ satisfies (\ref{cond}). Then
$\|\nph_k\|_{\min} = \|\Phi_k\|_2$. Let $\psi_n = \sum_{k=1}^{n}
\nph_k$. If $n < m$ then
 $\|\psi_m - \psi_n\|_{\min} = \|\sum_{k=n+1}^m
\Phi_k\|_2$ so
$$\|\psi_m
- \psi_n\|_{\min} \leq
\Big(\sum_{k=n+1}^{m}\|\Phi_k\|_2^2\Big)^{1/2}.$$
By (\ref{cond}), the sequence $\{\psi_n\}$ converges to an element
$\nph\in \C^{\dd}\otimes\C$. Moreover, for every $x\in \C_2(H)$ we
have
$$\Phi_{\id(\nph)}(x) = \lim_{n\rightarrow \infty}
p_n\Phi_{\id(\nph)}(x)p_n = \lim_{n\rightarrow \infty}
\Phi_{\id(\psi_n)}(x) = \Phi(x),$$ where the limits are in the
operator norm.
So $\Phi_{\id(\nph)} = \Phi$ which is completely contractive by
Lemma~\ref{saar}, so $\phi\in M(\C^\dd,\C)$ by Theorem~\ref{ecb}.
\end{proof}

Given C*-algebras $\A_i\subseteq\B(H_i)$, $i = 1,\dots,n$,
and $\psi=c_2\otimess c_{n-1}\in \A_2\odots\A_{n-1}$, we may
define a bounded linear map $\A_1\otimes\A_n\to \B_1\otimes
\A_2\otimess\A_n$, where $\B_1 = \A_1$ if $n$ is even and
$\B_1=\A_1^\dd$ if $n$ is odd, by
\[a\otimes b\mapsto
\begin{cases}
a\otimes \psi\otimes b\quad&\text{if $n$ is even,}\\
a^\dd\otimes \psi\otimes b\quad&\text{if $n$ is odd.}
\end{cases}
\]
We write $\iota_\psi$ for the restriction of this map to
$M(\A_1,\A_n)$.

\begin{lemma}\label{lem:iota}
  (i) The range of $\iota_{\psi}$ is contained in $M(\B_1,\A_2,\dots,\A_n)$.
  \smallskip

  (ii) $\iota_{\psi}(M_c^\id(\A_1,\A_{n}))\subseteq M_c^\id(\B_1,\A_2,\dots,\A_n)$.
  \smallskip

  (iii) Suppose that $n$ is even and $\omega\in (\B(H_{n-1}^{\dd})\ehaags
  \B(H_{2}))_*$.
  Writing \[M_{\omega}:\B(H_n)\ehaag \B(H_{n-1}^\dd)\ehaags
\B(H_2)\ehaag\B(H_1^\dd)\to \B(H_n)\ehaag\B(H_1^\dd) \]
for the ``middle slice map''
$M_\omega=R_\omega\ehaag\id_{\B(H_1^\dd)}$, we have
\[
M_\omega(u_{\iota_{\psi}(\phi)})=\omega(\tilde{\psi})u_\phi\]
where $\tilde{\psi} = c_{n-1}^\dd\otimess c_2$. The same is true,
mutatis mutandis, if $n$ is odd.
\end{lemma}
\begin{proof}
Let $\nph\in M(\A_1,\A_n)$. By Theorem~\ref{jtt}, there
exist a net $\{\varphi_\nu\}\subseteq {\A_1}\odot\A_n$ and
representations $u_{\varphi_\nu}^{\id}=A_2^{\nu}\odot A_1^\nu$ and
$u_\varphi^{\id}=A_2\odot A_1$, where $A_i^\nu$ are finite
matrices with entries in $\A_1^{\dd}$ if $i=1$ and in $\A_n$ if
$i=2$, such that $\varphi_\nu\to\varphi$ semi-weakly, $A_i^\nu\to
A_i$ strongly and $\sup_{i,\nu}\|A_i^\nu\| < \infty$.

\smallskip

(i) It is easy to see that $\iota_{\psi}(\phi_\nu)$ satisfies the
boundedness conditions of Theorem~\ref{jtt} and converges
semi-weakly to $\iota_{\psi}(\phi)$, which is therefore a
universal multiplier.

\smallskip

(ii)
  Suppose that $n$ is even and let $\iota=\iota_{\psi}$.  It is
  immediate to
  check that if $\varphi\in \A_1\odot\A_n$ and
  $T_1\in\K(H_1^{\dd},H_2)$, $\dots$, $T_{n-1}\in
  \K(H_{n-1}^{\dd},H_n)$ then
\begin{equation*}%
\Phi_{\iota(\varphi)}(T_{n-1}\otimes\ldots\otimes T_1)=
\Phi_{\varphi}(T_{n-1}c_{n-1}^{\dd}\ldots c_2T_1).
\end{equation*}
Note that %
this equation holds for any $\varphi\in
M(\A_1,\A_n)$ since $\Phi_{\varphi_\nu}(T)\to \Phi_{\varphi}(T)$
and $\Phi_{\iota(\varphi_\nu)}(T_{n-1}\otimes\ldots\otimes T_1)\to
\Phi_{\iota(\varphi)}(T_{n-1}\otimes\ldots\otimes T_1)$ weakly for
any $T$, $T_1,\ldots, T_{n-1}$. Since $\Phi_{\iota(\varphi)}$ is
the composition of the bounded mapping
$X_{n-1}\otimes\ldots\otimes X_1\mapsto X_{n-1}c_{n-1}^{\dd}\ldots
c_2X_1$ with $\Phi_{\varphi}$, it follows that if $\varphi$ is a
compact operator multiplier then so is $\iota(\varphi)$.

\smallskip

(iii)  We have that
\begin{align*}
\Phi_{\iota(\varphi_\nu)}(T_{n-1}\otimes\ldots\otimes
T_1)&=A_2^\nu(T_{n-1}\otimes 1)(c_{n-1}^{\dd}\otimes 1)\ldots
(c_2\otimes 1)(T_1\otimes 1)A_1^\nu\\&\to A_2(T_{n-1}\otimes
1)(c_{n-1}^{\dd}\otimes 1)\ldots(c_2\otimes 1)(T_1\otimes 1)A_1
\end{align*}
weakly. On the other hand,
$\Phi_{\iota(\varphi_\nu)}(T_{n-1}\otimes\ldots\otimes T_1)\to
\Phi_{\iota(\varphi)}(T_{n-1}\otimes\ldots\otimes T_1)$ which
implies that $u_{\iota(\varphi)}=A_2\odot (c_{n-1}^{\dd}\otimes
1)\odot\ldots\odot (c_2\otimes 1)\odot A_1$. It follows that
$M_\omega(u_{\iota(\varphi)})=\omega(\tilde{\psi})u_\varphi$.
\end{proof}

\begin{theorem}\label{gen}
Let $\As$ be C*-algebras with the property that both $\K(\A_1)$
and $\K(\A_n)$ contain full matrix algebras of arbitrarily large
sizes. Then the inclusion $M_{cc}(\As)\subseteq M_c(\As)$ is
proper.
\end{theorem}
\begin{proof} We may assume that $\A_i\subseteq\B(H_i)$, $i =
1,\dots,n$ for some Hilbert spaces $\Hs$. First suppose that $n =
2$ and let $H$ be an infinite dimensional separable Hilbert space
with $H^{\dd}\subseteq H_1$ and $H\subseteq H_2$. Let $\C\simeq
\bigoplus^{c_0}_{k\in\bb{N}} M_{k}$ be the C*-algebra from
Lemma~\ref{HS}. Then $\C^{\dd}\subseteq \A_1$ and $\C\subseteq
\A_2$. By the injectivity of the minimal tensor product of
C*-algebras, $\C^{\dd}\otimes\C\subseteq \A_1\otimes\A_2$.

Let $\nph\in \C^{\dd}\otimes\C$ be given by Lemma~\ref{HS}. It
follows from Lemma~\ref{saar} that $\nph\in
M_c(\A_1,\A_2)\setminus M_{cc}^\id(\A_1,\A_2)$. Since faithful
representations of $\A_1$ and $\A_2$ restrict to representations
of $\C$ containing the identity subrepresentation up to unitary
equivalence, we have that $\phi\in M_c(\A_1,\A_2)\setminus
M_{cc}(\A_1,\A_2)$.

Suppose now that $n$ is even. Let $\varphi\in
M_c(\A_1,\A_n)\setminus M_{cc}(\A_1,\A_n)$, fix any non-zero $\psi
= c_2\otimess c_{n-1}\in \A_2\odots\A_{n-1}$ and let us write
$\iota=\iota_{\psi}$.  Suppose that $\iota(\varphi)$ is a
completely compact multiplier. By Theorem~\ref{charcn},
$u_{\iota(\varphi)}\in \K(\A_n) \haag(\A_{n-1}^o\ehaags\A_2)\haag
\K(\A_1^o)$.

Let $\tilde{\psi} = c_{n-1}^\dd\otimess c_2\in
\A_{n-1}^\dd\ehaags\A_2$ and fix $\omega\in
(\B(H_{n-1}^{\dd})\ehaags \B(H_{2}))_*$ such that
$\omega(\tilde{\psi})\neq 0$. By Lemma~\ref{lem:iota}~(iii),
$M_{\omega}(u_{\iota(\varphi)})=\omega(\tilde{\psi})u_{\varphi}$
and hence $u_{\varphi}\in \K(\A_n)\haag \K(\A_1^o)$ which by
Theorem~\ref{charcn} contradicts the assumption that $\varphi$ is
not a completely compact multiplier.

If $n$ is odd then the same proof works with minor modifications.
\end{proof}

\begin{romanremark}
\label{rsmith}
We do not know whether the sets $M_{cc}(\A,\B)$ and $M_{c}(\A,\B)$ are
distinct if $\K(\A)$ contains matrix algebras of arbitrarily large
sizes, while $\K(\B)$ does not (and vice versa). 
To show that
the inclusion $M_{cc}(\C,c_0) \subseteq M_{c}(\C,c_0)$ is proper it
would suffice to exhibit mappings $\Phi_k : M_k\rightarrow M_k$ which
satisfy~(\ref{cond}) and are left $D_k$-modular (where $D_k$ is the
subalgebra of all diagonal matrices of $M_k$).  This modularity
condition would enable us to find $\nph_k\in M_k^\dd\otimes D_k$ such
that $\Phi_k=\Phi_{\id(\nph_k)}$ using the method of Lemma~\ref{HS}
and we could then conclude from Lemma~\ref{saar}
that $M_{cc}(\C,c_0)\subsetneq M_c(\C,c_0)$.
However, we do not know if such mappings $\Phi_k$ exist.
\end{romanremark}

\subsection{Automatic compactness}

We now turn to the question of when every universal multiplier is
automatically compact. We will restrict to the case $n = 2$ for
the rest of the paper. We will first establish an auxiliary result
in a different but related setting. Suppose that $\A$ and $\B$ are
commutative C*-algebras and assume that $\A = C_0(X)$ and $\B =
C_0(Y)$ for some locally compact Hausdorff spaces $X$ and $Y$. The
C*-algebra $C_0(X)\otimes C_0(Y)$ will be identified with
$C_0(X\times Y)$ and $M(\A,\B)$ with a subset of $C_0(X\times Y)$.
Elements of the Haagerup tensor product $C_0(X)\haag C_0(Y)$, as
well as of the projective tensor product
$C_0(X)\hat{\otimes}C_0(Y)$, will be identified with functions in
$C_0(X\times Y)$ in the natural way. Note that, by Grothendieck's
inequality, $C_0(X)\haag C_0(Y)$ and $C_0(X)\hat{\otimes}C_0(Y)$
coincide as sets of functions.

\begin{proposition}\label{co}
  Let $X$ and $Y$ be locally compact infinite Hausdorff spaces. Then
  $C_0(X)\haag C_0(Y)\subseteq M(C_0(X),C_0(Y))$ and this inclusion is
  proper.
\end{proposition}
\begin{proof} The inclusion $C_0(X)\haag C_0(Y)\subseteq
M(C_0(X),C_0(Y))$ follows from Corollary~6.7 of~\cite{ks}. To show
that this inclusion is proper, suppose first that $X$ and $Y$ are
compact. By Theorem~11.9.1 of~\cite{gmcg}, there exists a sequence
$(f_i)_{i=1}^{\infty}\subseteq C(X)\haag C(Y)$ such that
$\sup_{i\in\bb{N}}\|f_i\|_\hh < \infty$ converging uniformly to a
function $f\in C(X\times Y)\setminus C(X)\haag C(Y)$. By
Corollary~6.7 of~\cite{ks}, $f\in M(C(X),C(Y))$. The conclusion
now follows.

Now assume that both $X$ and $Y$ are locally compact but not
compact (the case where one of the spaces is compact while the
other is not is similar). Let $\tilde{X} = X\cup\{\infty\}$ and
$\tilde{Y} = Y\cup\{\infty\}$ be the one point compactifications
of $X$ and $Y$. Then $C(\tilde{X}) = C_0(X) + \bb{C}1$ and
$C(\tilde{Y}) = C_0(Y) + \bb{C}1$, where $1$ denotes the constant
function taking the value one. Moreover, it is easy to see that
$$C(\tilde{X})\otimes C(\tilde{Y}) = C_0(X\times Y) + C_0(X) +
C_0(Y) + \bb{C}1$$ and
\begin{equation}\label{proje}
C(\tilde{X})\hat{\otimes} C(\tilde{Y}) = C_0(X)\hat{\otimes}
C_0(Y) + C_0(X) + C_0(Y) + \bb{C}1.
\end{equation}
By the first part of the proof, there exists $\nph\in
M(C(\tilde{X}),C(\tilde{Y}))\setminus C(\tilde{X})\haag
C(\tilde{Y})$. Write $\nph = \nph_1 + \nph_2 + \nph_3 + \nph_4$
where $\nph_1 \in C_0(X\times Y)$, $\nph_2\in C_0(X)$, $\nph_3 \in
C_0(Y)$ and $\nph_4 \in \bb{C}1$. Suppose that $\nph_1\in
C_0(X)\haag C_0(Y)$. By (\ref{proje}), $\nph\in
C(\tilde{X})\hat{\otimes} C(\tilde{Y})$, a contradiction.
\end{proof}

\begin{theorem}\label{gen2}
Let $\A$ and $\B$ be C*-algebras. The following are equivalent:

\noindent
(i)\ \ either $\A$ is finite dimensional and $\K(\B) = \B$, or $\B$ is finite dimensional and $\K(\A) = \A$;

\noindent
(ii)\ $M_c(\A,\B) = M(\A,\B)$;

\noindent
(iii) $M_{cc}(\A,\B)=M(\A,\B)$.
\end{theorem}
\begin{proof} (i)$\Rightarrow$(iii) Suppose that $\A$ is finite
dimensional and $\K(\B) = \B$, and that $\A\subseteq \B(H_1)$ and
$\B\subseteq\B(H_2)$ for some Hilbert spaces $H_1$ and $H_2$ where
$H_1$ is finite dimensional. Fix $\nph\in M(\A,\B)$. Then $\nph$
is the sum of finitely many elements of the form $a\otimes b$
where $a$ has rank one and $b\in \K(H_2)$; such elements are
completely compact multipliers by Theorem~\ref{charcn}.\medskip

(iii)$\Rightarrow$(ii) is trivial.\medskip

(ii)$\Rightarrow$(i) Assume that both $\A$ and $\B$ are infinite
dimensional and are identified with their image under the reduced
atomic representation. If either $\K(\A)$ or $\K(\B)$ is finite
dimensional then there exists an elementary tensor $a\otimes b \in
(\A\odot\B) \setminus (\K(\A)\odot\K(\B))$. By
Proposition~\ref{ehaag}, $a\otimes b \not\in M_c(\A,\B)$. We can
therefore assume that both $\K(\A)$ and $\K(\B)$ are infinite
dimensional. Then, up to a $*$-isomorphism, $c_0$ is contained in
both $\K(\A)$ and $\K(\B)$. By Proposition~\ref{co}, there exists
$\nph\in M(c_0,c_0)\setminus (c_0\haag c_0)$. Then $\nph\in
M(\A,\B)$ and $\Phi_{\id(\nph)}$ is not compact by 
Hladnik's result~\cite{hladnik}.
Since the restrictions to
$c_0$ of any faithful representations of $\A$, $\B$ contain
representations unitarily equivalent to the identity
representations, we see that $\nph$ is not a compact multiplier.

Thus at least one of the C*-algebras $\A$ and $\B$ is finite
dimensional; assume without loss of generality that this is $\A$.
Suppose that $\B \neq \K(\B)$ and fix an element an $b\in
\B\setminus \K(\B)$. Let $a\in\A$ be a non-zero element. By
Proposition~\ref{ehaag}, the elementary tensor $a\otimes b$ is not
a compact multiplier.
\end{proof}

\end{document}